\pdfoutput=1
\documentclass[a4paper,10pt]{amsart}
 \usepackage{
  amsmath, amscd, amssymb, currfile,diagrams, mathrsfs,
  mathtools, fullpage, enumerate, thmtools,
  mdframed, fancyhdr, xspace,stmaryrd,
    tikz-cd,
  etoolbox, chngcntr,yfonts }
\usepackage[shortlabels]{enumitem}
 \setlist{leftmargin=2\parindent}

\usepackage[pagebackref, hypertexnames=false, bookmarks=false]{hyperref}

\usepackage[capitalise]{cleveref}

\newenvironment{smatrix}{\left( \begin{smallmatrix} } {\end{smallmatrix} \right) }

\newcommand{\stbt}[4]{\begin{smatrix}#1 & #2 \\ #3 & #4\end{smatrix}}

\theoremstyle{plain}
\newtheorem{theorem}{Theorem}[subsection]

\newtheorem{lemma}[theorem]{Lemma}
\newtheorem{proposition}[theorem]{Proposition}
\newtheorem{corollary}[theorem]{Corollary}
\newtheorem{definition}[theorem]{Definition}
\newtheorem{notation}[theorem]{Notation}

\newtheorem{assumption}[theorem]{Assumption}

\theoremstyle{remark}
\declaretheorem[name=Remark,sibling=theorem,qed={\lower-0.3ex\hbox{$\diamond$}}]{remark}
\declaretheorem[name=Note,sibling=theorem,qed={\lower-0.3ex\hbox{$\diamond$}}]{note}

\DeclareMathOperator{\Fil}{Fil}
\DeclareMathOperator{\GL}{GL}

\DeclareMathOperator{\GSp}{GSp}

\DeclareMathOperator{\Gal}{Gal}
\DeclareMathOperator{\Gr}{Gr}

\DeclareMathOperator{\Hom}{Hom}

\DeclareMathOperator{\Iw}{Iw}
\DeclareMathOperator{\Kl}{Kl}
\DeclareMathOperator{\LE}{\mathcal{LE}}

\DeclareMathOperator{\NNfp}{NN-fp}
\DeclareMathOperator{\NNsyn}{NN-syn}

\DeclareMathOperator{\Sieg}{Si}
\DeclareMathOperator{\Spec}{Spec}
\DeclareMathOperator{\Sym}{Sym}

\DeclareMathOperator{\coh}{coh}

\DeclareMathOperator{\crit}{crit}

\DeclareMathOperator{\diag}{diag}
\DeclareMathOperator{\fp}{fp}

\DeclareMathOperator{\lrigfp}{lrig-fp}
\DeclareMathOperator{\lrigsyn}{lrig-syn}

\DeclareMathOperator{\mot}{mot}

\DeclareMathOperator{\pr}{pr}
\DeclareMathOperator{\rigfp}{rig-fp}

\DeclareMathOperator{\sph}{sph}
\DeclareMathOperator{\st}{st}
\DeclareMathOperator{\syn}{syn}
\DeclareMathOperator{\tr}{tr}

\newcommand{\f}{\mathrm{f}}

\renewcommand{\AA}{\mathbf{A}}

\newcommand{\Gh}{\widehat{G}}
\newcommand{\Gt}{\widetilde{G}}

\newcommand{\fa}{\mathfrak{a}}
\newcommand{\fb}{\mathfrak{b}}

\newcommand{\Af}{\AA_{\mathrm{f}}}

\newcommand{\cc}{\mathbf{c}}
\newcommand{\Dcris}{\mathbf{D}_{\mathrm{cris}}}
\newcommand{\DdR}{\mathbf{D}_{\mathrm{dR}}}

\newcommand{\Eis}{\mathrm{Eis}}

\newcommand{\ucG}{\underline{\mathcal{G}}}

\newcommand{\Pif}{\Pi_{\f}}
\newcommand{\Sif}{\Sigma_{i,\f}}

\newcommand{\QQbar}{\overline{\QQ}}
\newcommand{\QQ}{\mathbf{Q}}
\newcommand{\RR}{\mathbf{R}}

\newcommand{\Qp}{\QQ_p}

\newcommand{\ZZ}{\mathbf{Z}}
\newcommand{\Zp}{\ZZ_p}

\newcommand{\VV}{\mathbf{V}}

\newcommand{\cD}{\mathcal{D}}
\newcommand{\cE}{\mathcal{E}}

\newcommand{\cG}{\mathcal{G}}
\newcommand{\cH}{\mathcal{H}}

\newcommand{\cL}{\mathcal{L}}

\newcommand{\cN}{\mathcal{N}}
\newcommand{\cO}{\mathcal{O}}

\newcommand{\cS}{\mathcal{S}}

\newcommand{\cV}{\mathcal{V}}
\newcommand{\cW}{\mathcal{W}}
\newcommand{\cX}{\mathcal{X}}
\newcommand{\cY}{\mathcal{Y}}

\newcommand{\hY}{\hat{Y}}

\newcommand{\hiota}{\hat\iota}
\newcommand{\can}{\mathrm{can}}
\newcommand{\ch}{\mathrm{ch}}
\newcommand{\dep}{\mathrm{dep}}
\newcommand{\dR}{\mathrm{dR}}
\newcommand{\et}{\text{\textup{\'et}}}

\newcommand{\into}{\hookrightarrow}

\newcommand{\ord}{\mathrm{ord}}
\newcommand{\res}{\mathrm{res}}

\newcommand{\sFil}{{\mathscr Fil}}
\newcommand{\sF}{\mathscr{F}}

\newcommand{\sH}{\mathscr{H}}

\newcommand{\sW}{\mathscr{W}}

\newcommand{\wH}{\widetilde{H}}

\numberwithin{equation}{section}
\renewcommand{\le}{\leqslant}
\renewcommand{\leq}{\leqslant}
\renewcommand{\ge}{\geqslant}
\renewcommand{\geq}{\geqslant}

\author{David Loeffler}

\author{Sarah Livia Zerbes}

\title{On $p$-adic regulators for $\GSp(4)\times \GL(2)$ and $\GSp(4)\times \GL(2) \times \GL(2)$}

\thanks{Supported by the following grants: Royal Society University Research Fellowship ``$L$-functions and Iwasawa theory'' and EPSRC Standard Grant EP/S020977/1 (Loeffler); ERC Consolidator Grant ``Euler systems and the Birch--Swinnerton-Dyer conjecture'' (Zerbes).}

\begin{document}
 \renewcommand{\crefrangeconjunction}{--} 


\begin{abstract}
 We compute the images of motivic cohomology classes for $\GSp_4\times\GL_2$ and $\GSp_4 \times \GL_2 \times \GL_2$ under the syntomic regulator, and relate them to non-critical values of suitable $p$-adic $L$-functions.
\end{abstract}

 \maketitle


 \makeatletter
 \patchcmd{\@tocline}
 {\hfil}
 {\leaders\hbox{\,.\,}\hfil}
 {}{}
 \makeatother

 \setcounter{tocdepth}{1}
 \tableofcontents

\section{Introduction}

  The main results of \cite{HJS20} and \cite[\S 8]{LZvista} concern the construction of classes in the motivic cohomology of the Shimura varieties associated to the groups  $\GSp_4\times\GL_2$ and $\GSp_4 \times \GL_2 \times \GL_2$. These classes are obtained as the pushforward of suitable classes in the motivic cohomology of the product of two modular curves, similar to the construction of an Euler system for the Shimura variety of $\GSp_4$ in \cite{LSZ17}. They can be shown to vary in $p$-adic families (in the former case, they form an Euler system, while the latter construction the vary in Hida families, analogous to the construction of  diagonal cycles in \cite{darmonrotger14}), and they hence have potential applications to the Bloch--Kato conjecture and the Iwasawa Main Conjecture for the Galois representations attached to automorphic forms of the underlying Shimura variety. The missing piece of information for these arithmetic applications is  a so-called \emph{explicit reciprocity law},  relating the cohomology classes to critical values of the $L$-functions of the automorphic representations.

  In \cite{LZ20}, we proved such an explicit reciprocity law for the Euler system attached to the Shimura variety of $\GSp(4)$. This strategy of the proof consisted roughly of two steps: firstly, to compute the image of the Euler system under the syntomic regulator, identifying it with a non-critical value of a $p$-adic $L$-function, and secondly, to use $p$-adic deformation to  relate the Euler system to critical values of the $p$-adic $L$-function.

  In this paper, we carry out the first step for the cohomology classes mentioned above. The strategy of the computation is parallel to that in the $\GSp(4)$-case; the only fundamental difference is the replacement of one (resp. both) $\GL_2$-Eisenstein classes in the $\GSp(4)$-case by one (resp. two) classes arising from $\GL_2$-cusp forms. As in the $\GSp(4)$ case, our arguments are conditional on a vanishing statement for certain eigenspaces in rigid cohomology (Conjecture 10.2.3 in \cite{LZ20}), which will be treated in forthcoming work of Lan and Skinner. The question of relating the cohomology classes to critical values of the $p$-adic $L$-function, as well as the arithmetic applications, will be the topic of a forthcoming paper.

\subsubsection*{Acknowledgements.} We would like to thank Chi-Yun Hsu, Zhaorong Jin and Ryotaro Sakamoto for answering our questions about their work on the $\GSp_4\times\GL_2$ Euler system.



\section{Preliminaries}\label{sec:prelim}

 \subsection{Conventions}\label{ss:conventions}

  In this paper, $p$ is a prime. As in \cite[\S 2]{LSZ17}, $G$ denotes the symplectic group $\GSp_4$, $P_{\Sieg}$ and $P_{\Kl}$ denote its standard Siegel and Klingen parabolic subgroups. Define the groups
  \[ \Gh = G \times \GL_2\qquad \text{and}\qquad \Gt=G\times \GL_2\times \GL_2, \]
  where the map $G\rightarrow \GL_1$ is the symplectic multiplier and $\GL_2\rightarrow \GL_1$ is the determinant.  Let $H$ denote the group $\GL_2 \times_{\GL_1} \GL_2$. We consider $H$ as a subgroup of $G$, $\Gh$ and $\Gt$ via the embeddings
  \begin{align*}
    \iota: H\hookrightarrow G,\quad &\left[\begin{pmatrix} a & b \\ c & d\end{pmatrix},
    \begin{pmatrix} a' & b'\\ c'& d'\end{pmatrix}\right]\mapsto
    \left[ \begin{smatrix} a &&& b\\ & a' & b' & \\ & c' & d' & \\ c &&& d \end{smatrix}\right], \\
    \hat\iota:H\hookrightarrow \Gh,\quad & (h_1,h_2)\mapsto \left(\iota(h_1,h_2),\, h_2\right),\\
    \tilde\iota: H\hookrightarrow \Gt,\quad & (h_1,h_2)\mapsto \left(\iota(h_1,h_2),\, h_1,h_2\right)
  \end{align*}

  \subsubsection*{Characters} If $\chi: (\ZZ / N)^\times \to A^\times$ is a character, for some ring $A$, then we let $\widehat{\chi}$ denote the character $\AA^\times / \QQ^\times\RR^{\times}_{> 0} \to A^\times$ such that $\widehat{\chi}(\varpi_\ell) = \chi(\ell)$ for primes $\ell \nmid N$, where $\varpi_\ell$ is a uniformizer at $\ell$. Note that the restriction of $\widehat{\chi}$ to $\widehat{\ZZ}^\times \subset \Af^\times$ is given by $z \mapsto \chi(z^{-1} \bmod N)$.


\subsection{Branching laws}
 \label{ss:branching1}

 \begin{definition} \
  \label{def:repsHandV}
  \begin{enumerate}[(a)]
   \item Write $\cH$ for the standard representation of $\GL_2$, following \cite[Def. 7.6.1]{LZ20}. {\em (Note that this is the opposite of the convention in \cite{LSZ17}.)}
   \item For $0\leq r_2\leq r_1$, write $V(r_1, r_2)$ for the representation $V(r_1,r_2;r_1+r_2)$ of $G$ (with conventions as in \cite{LPSZ1}).
  \end{enumerate}
 \end{definition}

 \begin{proposition}
  \label{prop:branchingG}

  Let $r_1 \ge r_2 \ge 0$ be integers, and let $t_1, t_2$ be integers with $t_1 + t_2 = r_1 + r_2 \bmod 2$ and
  \[ |t_1 - t_2| \le r_1 - r_2, \qquad r_1 - r_2 \le t_1 + t_2 \le r_1 + r_2.\]

  Then:

  \begin{enumerate}[(a)]
   \item If $V$ denotes the representation $V(r_1, r_2)$ of $G$, then we have
   \begin{align*}
    \Sym^{t_1}\cH^\vee\boxtimes \Sym^{t_2}\cH^\vee \ &\subset\  \iota^*\left(V^\vee\otimes\mu^q\right),\\
    \intertext{
     where $q = \tfrac{1}{2}(r_1 + r_2 -t_1 - t_2) \in \ZZ_{\ge 0}$.
     \item If $\hat{V}$ denotes the representation $V \boxtimes \Sym^{t_2} \cH$ of $\Gh$, then
    }
   \Sym^{t_1}\cH^\vee \boxtimes \mathbf{1} \ &\subset\  \hat{\iota}^*\left(\hat{V}^\vee\otimes\mu^{(q + t_2)}\right).\\
   \intertext{
    \item If $\tilde{V}$ denotes the representation $V\boxtimes \Sym^{t_1}\cH\boxtimes \Sym^{t_2}\cH$ of $\Gt$, then
   }
   \mathbf{1} \boxtimes \mathbf{1} \ &\subset\
    \tilde\iota^*\left(\tilde{V}^\vee\otimes\mu{}^{(q + t_1 + t_2)}\right).
   \end{align*}
  \end{enumerate}
  In each case, the subrepresentation appears with multiplicity one.
 \end{proposition}

 \begin{proof}
  Statement (a) is \cite[Prop. 6.4]{LPSZ1}. The other statements follow readily from this, noting that $\cH^\vee = \cH \otimes \det^{-1}$. (See also \cite[Cor. 3.2]{HJS20}.)
 \end{proof}

 \begin{remark}
  Note that these inequalities correspond to region (e) in \cite[Figure 2]{LZvista}. Conversely, one can check that an irreducible representation of $\Gt$ has a vector invariant under $\tilde{\iota}(H)$ if and only if it has the form given in (c), for some $r_1, r_2, t_1, t_2$ satisfying the stated inequalities (up to a twist by a character trivial on the image of $H$). Similar statements can be formulated for (a) and (b). (We will not use these statements directly, but they show that our statements are in some sense ``optimal''.)
 \end{remark}


 \subsection{Pushforward and pullback maps}

  We describe here the construction of motivic cohomology classes for the Shimura varieties attached to $G$, $\Gh$ and $\Gt$. The construction for $G$ is identical to \cite{LSZ17}, but we recall it here in parallel with the other two cases, in order to show the common aspects of the three constructions.

  \begin{definition}
   For $\star\in\{G,\Gh,\Gt, H,\GL_2\}$, write $Y_{\star}$ for the canonical $\QQ$-model of the Shimura variety attached to $\star$, and $Y_{\star, L}$ for its base-extension to a field $L \supseteq \QQ$.
  \end{definition}

  \begin{notation} Let $(r_1, r_2, t_1, t_2)$ be integers satisfying the conditions of \cref{prop:branchingG}.
   \begin{itemize}
    \item Denote by $\sH_{\QQ}$ the $\GL_2(\Af)$-equivariant relative Chow motive over $Y_{\GL_2}$ associated to the standard representation $\cH$ of $\GL_2$.

    \item Write  $\cV_{\QQ}$ for the $G(\AA_f)$-equivariant relative Chow motive of $V$ over $Y_G$.

    \item Write $\hat\cV_{\QQ}$ and $\tilde\cV_{\QQ}$ for the Chow motives associated to $\hat{V}$ and $\tilde{V}$ over $\hat{Y}$ and $\tilde{Y}$, respectively.

   \end{itemize}
  \end{notation}

  \begin{proposition}\label{lem:goodsubrep}
   Let $(r_1, r_2, t_1, t_2)$ be integers satisfying the conditions of \cref{prop:branchingG}. Then we have the following pushforward maps in motivic cohomology:
   \begin{subequations}
   \begin{align}
    \iota_*: H^2_{\mot}\left(Y_H(K \cap H), (\Sym^{t_1}\sH^\vee_{\QQ}\boxtimes\Sym^{t_2}\sH^\vee_{\QQ}) (2)\right)
    &\longrightarrow H^4_{\mot}\left(Y_G(K), \cV_{\QQ}^\vee(3-q)\right),\\
    \hat{\iota}_*: H^1_{\mot}\left(Y_H(\hat{K} \cap H), (\Sym^{t_1}\sH^\vee_{\QQ}\boxtimes \mathbf{1})(1)\right)
    &\longrightarrow H^5_{\mot}\left(Y_{\Gh}(\hat{K}),\hat\cV_{\QQ}^\vee(3-q-t_2)\right),\\
    \tilde{\iota}_*: H^0_{\mot}\left(Y_H(\tilde{K} \cap H), \QQ(0)\right)
    &\longrightarrow H^6_{\mot}\left(Y_{\Gt}(\tilde{K}), \tilde{\cV}_{\QQ}^\vee(3-q-t_1-t_2)\right).
   \end{align}
   \end{subequations}
   Here $K, \hat{K}, \tilde{K}$ are sufficiently small open compact subgroups of $G(\Af)$, $\Gh(\Af)$ and $\Gt(\Af)$ respectively.
  \end{proposition}

 \subsection{Lemma--Eisenstein maps}\label{ss:ESclass}

  \begin{notation}
   We let $\cS(\Af^2)$ denote the space of $\QQ$-valued Schwartz functions on $\Af^2$, with $\GL_2(\Af)$ acting by right-translation. We write $\cS_0(\Af^2)$ for the subspace of functions such that $\Phi(0, 0) = 0$.
  \end{notation}

  For any $t \in \ZZ_{\ge 0}$ we have a map, the \emph{motivic Eisenstein symbol} (c.f. \cite[\S 4.1]{KLZ20}),
  \[ \Eis^t_{\mot}: \cS_{(0)}(\Af^2) \to H^1_{\mot}\left(Y_{\GL_2}, \Sym^t(\sH^\vee_{\QQ})(1)\right), \]
  where $Y_{\GL_2}$ denotes the direct limit over all levels, and $\cS_{(0)}$ is read as $\cS_0$ if $t = 0$ and $\cS$ if $t \ge 1$. Thus we may compose the maps of \cref{lem:goodsubrep} with 2, 1 or 0 copies of the Eisenstein symbol, after normalising by appropriate volume factors to give a map which is compatible with pullback in the level.

  This gives maps
  \begin{subequations}
   \begin{align}
    \LE_{\mot}: \cS_{(0, 0)}(\Af^2 \times \Af^2) \otimes \cH(G(\Af))
    &\longrightarrow H^4_{\mot}\left(Y_G(K), \cV_{\QQ}^\vee(3-q)\right)[-q],\\
    \widehat{\LE}_{\mot}: \cS_{(0)}(\Af^2) \otimes \cH(\Gh(\Af))
    &\longrightarrow H^5_{\mot}\left(Y_{\Gh},\hat\cV_{\QQ}^\vee(3-q-t_2)\right)[-q-t_2],\\
    \widetilde{\LE}_{\mot}: \cH(\Gt(\Af))
    &\longrightarrow H^6_{\mot}\left(Y_{\Gt}, \tilde{\cV}_{\QQ}^\vee(3-q-t_1-t_2)\right)[-q-t_1-t_2].
   \end{align}
  \end{subequations}

  Here $\cS_{(0, 0)}(\Af^2 \times \Af^2)$ denotes functions vanishing identically along $\{(0, 0)\} \times \Af^2$ if $t_1 = 0$, and along $\Af^2 \times \{(0, 0)\}$ if $t_2 = 0$; and in the second formula, $\cS_{(0)}(\Af^2)$ denotes $\cS$ if $t_1 = 0$ and no restriction otherwise. The notation $[m]$ denotes twisting by the character $\|\mu\|^m$ of $G(\Af)$, as in \cite[\S 6.2]{LSZ17}. These maps satisfy equivariance properties for the actions of $H(\Af) \times G(\Af)$, $H(\Af) \times \Gh(\Af)$ and $H(\Af) \times \widetilde{G}(\Af)$ respectively, which we shall not spell out explicitly here; the case of $\LE_{\mot}$ is described in detail in \cite[\S 8.2]{LSZ17}, and the other two cases are similar.

  \begin{remark} See \cite{HJS20} for the case of $\widehat{G}$; the map denoted $\widehat{\LE}_{\mot}$ here is the ``symbol map'' of \emph{op.cit.}.
  \end{remark}



\section{Automorphic representations and Galois cohomology}
\label{sect:ES}

 \subsection{The setting}

  The results of this paper will concern the following three settings:
  \begin{enumerate}
   \item[(A)] $\Pi$ is a cuspidal automorphic representation of $G$; and we study the 4-dimensional spin Galois representation attached to $\Pi$, and its character twists, using the map $\LE_{\mot}$.

   \item[(B)] $\Pi$ is as in (A), and $\Sigma_2$ is a cuspidal automorphic representation of $\GL_2$; and we study the 8-dimensional Galois representation associated to $\widehat{\Pi} = \Pi \boxtimes \Sigma_2$, viewed as a cuspidal automorphic representation of $\Gh$, using the map $\widehat{\LE}_{\mot}$.

   \item[(C)] $\Pi$ is as in (A), and $\Sigma_1, \Sigma_2$ are cuspidal automorphic representations of $\GL_2$, such that the product of the central characters of $\Pi, \Sigma_1$ and $\Sigma_2$ is trivial; and we study the 16-dimensional self-dual Galois representation associated to $\Pi \boxtimes \Sigma_1 \boxtimes \Sigma_2$, viewed as a cuspidal automorphic representation of $\Gt$, using the map $\widetilde{\LE}_{\mot}$.
  \end{enumerate}

  The goal of this section is to define, in each of the above cases, a map $\operatorname{Reg}^p_\nu: \mathscr{T}^p \to L$, where $p$ is a prime where our automorphic representations are unramified, $L$ is a finite extension of $\Qp$, and $\mathscr{T}^p$ is a suitable space of ``test data'' away from $p$, depending on $\Pi$ and the $\Sigma_i$. By construction, this map measures the local non-triviality of a global Galois cohomology class. In the following sections, we shall relate this map $\operatorname{Reg}^p_\nu$ to cup-products in coherent cohomology, which we shall interpret as special values of the $p$-adic $L$-functions of \cite{LPSZ1} and \cite{LZvista}.

  The construction of $\operatorname{Reg}^p_\mu$ in case (A) is already given in \cite{LZ20} (building on the results of \cite{LSZ17}). However, we shall still include case (A) in our account here, in order to clarify the relationship between this construction and the new results we are proving in case (B) and (C).

  \begin{remark}
   To some extent, cases (A) and (B) can be seen as ``degenerate cases'' of case (C), with $\Sigma_1$ in case (B) and both of the $\Sigma_i$ in case (A) replaced by spaces of Eisenstein series. However, there are limitations to this analogy: the definition of motivic classes is a little different between the three cases -- there does not seem to be a way to recover the maps $\LE_{\mot}$ or $\widehat{\LE}_{\mot}$ from $\widetilde{\LE}_{\mot}$.
  \end{remark}

 \subsection{Automorphic representations for $G$}
  \label{sect:heckeparams}

  \begin{definition}
   Let $(\Pi^H, \Pi^W)$ be a pair of non-endoscopic, non-CAP automorphic representations of $G(\AA_{\QQ})$ with the same finite part $\Pif$, as in \cite{LSZ17}, with $\Pi^W$ globally generic, discrete series at $\infty$ of weight $(k_1, k_2) = (r_1 + 3, r_2 + 3)$ for some integers $r_1 \ge r_2 \ge 0$. Let $\chi_{\Pi}$ be the Dirichlet character such that $\Pi$ has central character $\widehat{\chi}_{\Pi}$; note that $\chi_{\Pi}(-1) = (-1)^{r_1 + r_2}$.
  \end{definition}

  \begin{definition}
   We write $\Pif'$ for the ``arithmetically normalised'' twist $\Pif \otimes \|\cdot\|^{-(r_1 + r_2)/2}$, which is definable over a number field $E$. For any field $F \supseteq E$, we let $\mathscr{W}(\Pi_{\f}')_F$ be the vectors defined over $F$ in the Whittaker model of $\Pif'$, as in \cite{LZ20}.
  \end{definition}

  \begin{definition}
   Let $p$ be a prime such that $\Pi_p$ is unramified. We write $\alpha, \beta, \gamma, \delta$ for the Hecke parameters of $\Pi_{p}'$, and $P_p(X)$ for the polynomial $(1 - \alpha X) \dots (1 - \delta X)$.
  \end{definition}

  The Hecke parameters are algebraic integers in $\bar{E}$, and are well-defined up to the action of the Weyl group. Extending $E$ if necessary, we may assume that they lie in $E$ itself. They all have complex absolute value $p^{(r_1 + r_2 + 3)/2}$, and they satisfy $\alpha \delta = \beta \gamma = p^{(r_1 + r_2 + 3)} \chi_\Pi(p)$ (and $\chi_{\Pi}(p)$ is a root of unity).

  \begin{note}
   The polynomial $P_p(X)$ is consistent with the notation of Theorem 10.1.3 of \cite{LSZ17}, and in particular the local $L$-factor is given by
   \[ L(\Pi_p, s - \tfrac{r_1+r_2+3}{2}) = P_p\left(p^{-s}\right)^{-1} = \left[\left (1 - \tfrac{\alpha}{p^{s}}\right) \dots\right]^{-1}.\qedhere\]
  \end{note}

  We shall fix an embedding $E \into L \subset \QQbar_p$, where $L$ is a finite extension of $\Qp$, and let $v_p$ be the valuation on $L$ such that $v_p(p) = 1$. If we order $(\alpha, \beta, \gamma, \delta)$ in such a way that $v_p(\alpha) \le \dots \le v_p(\delta)$ (which is always possible using the action of the Weyl group), then we have the valuation estimates
  \[ v_p(\alpha) \ge 0,\qquad v_p(\alpha \beta) \ge r_2 + 1. \]

  \begin{definition}
   We say $\Pi$ is \emph{Siegel ordinary} at $p$ if $v_p(\alpha) = 0$, and \emph{Klingen ordinary} at $p$ if $v_p(\alpha\beta) = r_2 + 1$ (and \emph{Borel ordinary} if it is both Siegel and Klingen ordinary).
  \end{definition}

  Taking $V = V(r_1, r_2; r_1 + r_2)$ as in Definition \ref{def:repsHandV}, for each level $U \subset G(\Af)$, the $\Pif'$-isotypical part of $H^3_{\et,c}(Y_G(U)_{\QQbar}, \cV) \otimes_{\Qp} L$ is isomorphic to the sum of $\dim\left( \Pif^U\right)$ copies of a 4-dimensional $L$-linear Galois representation, uniquely determined up to isomorphism, whose semisimplification is the representation $\rho_{\Pi, p}$ associated to $\Pi$ \cite[\S\S 10.1-10.2]{LSZ17}. This is characterised by the relation
  \[ \det\left(1 - X \rho_{\Pi,p}(\operatorname{Frob}_\ell^{-1})\right) = P_\ell(X)\]
  for good primes $\ell$, where $\operatorname{Frob}_\ell$ is an arithmetic Frobenius at $\ell$.

  \begin{definition}
   We define
   \[ V_p(\Pi) \coloneqq \Hom_{L[G(\Af)]}\Big(\mathscr{W}(\Pi_{\f}')_L, H^3_{\et, c}( Y_{G, \QQbar}, \cV_L) \Big).\]
  \end{definition}

  This is a 4-dimensional $L$-linear representation of $\Gal(\overline{\QQ}/\QQ)$ which is a canonical representative of the isomorphism class $\rho_{\Pi, p}$. We therefore obtain a canonical isomorphism
  \[
   \sW(\Pif')_L = \Hom_{\Gal(\QQbar/\QQ)}\Big(V_p(\Pi), H^3_{\et,c}(Y_{G,\QQbar}, \cV_L)[\Pif']\Big).
  \]
  Using the canonical\footnote{This duality depends on a choice of Haar measure on $G(\Af)$; we choose this such that $G(\widehat{\ZZ})$ has volume 1.} duality of $G(\Af) \times \Gal(\QQbar/\QQ)$-representations
  \[
  \big\langle\!\!\big\langle -, -\big\rangle\!\!\big\rangle_G:  \Big(H^3_{\et,c}(Y_{G,\QQbar}, \cV_L)\Big) \times \Big(H^3_{\et}(Y_{G,\QQbar}, \cV_L^\vee(3))\Big) \to L,
  \]
  we can interpret elements of $\sW(\Pif')_L$ as homomorphisms of Galois representations $H^3_{\et}(Y_{G,\QQbar}, \cV_L^\vee(3)) \to V_p(\Pi)^*$, i.e.~as \emph{modular parametrisations} of $V_p(\Pi)^*$ in the sense of \cite[\S 10.4]{LSZ17}.


 \subsection{Automorphic representations for $\GL_2$}
  \label{sect:autoGL2}

  \begin{definition}
   For $i= 1, 2$ in case (C), and for $i = 2$ in case (B), let $\Sigma_i$ be an automorphic representation of $\GL_2(\AA_{\QQ})$ generated by a holomorphic cuspidal newform $g_i$ of weight $t_i+2\geq 2$. Let $\Sigma_i' = \Sigma_i \otimes \|\cdot\|^{-t_i/2}$, so that $\Sif'$ is defined over a number field.

   We let $\chi_{\Sigma_i}$ denote the nebentype character of $g_i$, so that $\Sigma_i$ has central character $\widehat{\chi}_{\Sigma_i}$.
  \end{definition}

  Replacing $E$ by a finite extension if necessary, we may assume that it is a common field of definition for both $\Pif'$ and $\Sif'$.

  \begin{definition}
   Let $p$ be a prime such that $\Sif$ is unramified at $p$. We write $\fa_i,\, \fb_i$ for the Hecke parameters of $\Sigma_{i,p}'$, and $Q_{p, i}(X)$ for the polynomial $(1 - \fa_i X)(1 - \fb_i X)$.
  \end{definition}

  The Hecke parameters are algebraic integers in $\bar{E}$, well-defined up to ordering. Extending $E$ if necessary, we may assume that they  lie in $E$ itself. They both have complex absolute value $p^{(t_i+1)/2}$,and they satisfy $\fa_i + \fb_i = a_p(g_i)$ and $\fa_i\fb_i = p^{(t_i+1)} \chi_{\Sigma_i}(p)$.

  \begin{note}
   The  local $L$-factor is given by
   \[ L(\Sigma_{i,p}, s - \tfrac{t_i+1}{2}) = Q_{p,i}\left(p^{s}\right)^{-1} = \left[\left(1 - \tfrac{\fa_i}{p^{s}}\right)\left(1 - \tfrac{\fb_i}{p^{s}}\right)\right]^{-1}.\qedhere\]
  \end{note}


  By a similar construction to the above, we can attach to $\Sigma_i$ a $2$-dimensional $L$-linear Galois representation $V_p(\Sigma_i)$, and we can interpret the $L$-linear Whittaker model $\sW(\Sif')_L$ as a space of homomorphisms of Galois representations $V_p(\Sigma_i) \to H^1_{\et,c}(Y_{\GL_2,\QQbar}, \Sym^{t_i}\sH_L)$, or dually $H^1_{\et}(Y_{\GL_2,\QQbar}, \Sym^{t_i}\sH_L^\vee(1)) \to V_p(\Sigma_i)^*$.

  \begin{remark}
   Note that $\sW(\Sif')_L$ has a canonical nonzero vector (the normalised Whittaker newform), which corresponds to the realisation of $V_p(\Sigma_i)$ as the $\Sif'$-isotypical part of cohomology at level $K_1(N_i) = \{g: g = \stbt{\star}{\star}{0}{1} \bmod N_i\}$ where $N_i$ is the conductor of $\Sigma_i$. Similar remarks apply to $V_p(\Pi)$, using the newvector theory of \cite{robertsschmidt07} and \cite{okazaki}.
  \end{remark}

  \begin{proposition}
   There is a canonical vector $\nu_{i,\dR} \in \Fil^1 \DdR( V_p(\Sigma_i))$ characterised as follows: we have an isomorphism
   \[ \Fil^1 \DdR( V_p(\Sigma_i)) = \Hom_{\GL_2(\Af)}\left(\sW(\Sigma'_{i,\f})_E, H^0(X_{\GL_2, E},\omega^{(t_i +2)}) \right) \otimes_E L, \]
   and $\nu_{i,\dR}$ maps the normalised new-vector $w_i^{\mathrm{new}} \in \sW(\Sigma'_{i,\f})_E$ to $G(\chi_{\Sigma_i}^{-1}) \cdot g_i$, where $g_i$ is the normalised newform generating $\Sigma_i'$.
  \end{proposition}

  In other words, the composite of $\nu_{i,\dR}$ with the $q$-expansion map, sending a modular form to its associated Whittaker function, is multiplication by $G(\chi_{\Sigma_i}^{-1})$. (The presence of the Gauss sum is a consequence of the fact that the cusp $\infty$ is not rational in our model of $Y_1(N) = Y_{\GL_2}(K_1(N))$; see \cite[\S 6.1]{KLZ20}.)


 \subsection{Eisenstein representations}

  We make the following definitions in cases (A) and (B), which are slightly artificial but allow us to treat them on the same footing as (C).

  \begin{notation}
   For $i = 1$ in case (B), and $i \in \{1, 2\}$ in case (A), let $\chi_{\Sigma_i}$ be a Dirichlet character, and $t_i \ge 0$ an integer, with $(-1)^{t_i} = \chi_{\Sigma_i}(-1)$.
  \end{notation}

  We thus have integers $(r_1, r_2, t_1, t_2)$, and Dirichlet characters $(\chi_{\Pi}, \chi_{\Sigma_1}, \chi_{\Sigma_2})$, in all three cases.

  \begin{assumption}
   We suppose that $\chi_{\Pi} \cdot \chi_{\Sigma_1} \cdot \chi_{\Sigma_2} = 1$ (implying that $r_1 + r_2 = t_1 + t_2 \bmod 2$), and that $(r_1, r_2, t_1, t_2)$ satisfy the inequalities of \cref{prop:branchingG}.
  \end{assumption}

  \begin{definition}\label{def:aibiEis}
   If $\Sigma_i$ is Eisenstein, we set
   \[ V_p(\Sigma_i) \coloneqq L(-1-t_i), \qquad \fa_i = p^{t_i + 1}, \qquad \fb_i = \chi_{\Sigma_i}(p), \]
   and we let $\nu^{(i)}_{\dR}$ denote the canonical basis vector of $\Fil^{(1 + t_i)} \DdR(L(-1-t_i)) = L$.
  \end{definition}

  \begin{remark}
   If $t_i > 0$, then there is a non-cuspidal, discrete series automorphic representation $\Sigma_i'$ of $\GL_2(\AA)$, generated by holomorphic Eisenstein series of weight $t_i + 2$, whose central character is $\hat\chi_{\Sigma_i} \|\cdot\|^{-t_i}$ and whose $T_\ell$-eigenvalue at a prime $\ell \nmid \operatorname{cond}(\chi_{\Sigma_i})$ is given by $\chi_{\Sigma_i}(\ell) + \ell^{(t_i + 1)}$. We can think of these Eisenstein representations as playing the role of the missing cuspidal automorphic representations in cases (A) and (B). If $t_i = 0$ the situation is more complicated, since the corresponding Eisenstein series are only nearly-holomorphic rather than holomorphic, and the representation they generate is not irreducible; but we are only using $\Sigma_i$ as a notational device anyway, so this will not concern us greatly.
  \end{remark}

 \subsection{Tensor products}

  \begin{definition}
   Let $h = 2 + q + t_1 + t_2 = \tfrac{r_1 + r_2 + t_1 + t_2 + 4}{2}$, and define a Galois representation $\VV$ by
   \[ \VV = V_p(\Pi) \otimes V_p(\Sigma_1)\otimes V_p(\Sigma_2), \]
   with the above interpretation of $V_p(\Sigma_i)$ if $\Sigma_i$ is Eisenstein.
  \end{definition}

   Using the above constructions, together with the K\"unneth formula, we can view $\sW(\Pif')_L$ in case (A), $\sW(\Pif')_L \otimes \sW(\Sigma_{2, \f}')$ in case (B), and $\sW(\Pif')_L \otimes \sW(\Sigma_{1, \f}')\otimes \sW(\Sigma_{2, \f}')$  in case (C), as homomorphisms of Galois representations
   \begin{align*}
    H^3_{\et}\left(Y_{G,\QQbar}, \cV^\vee(3-q)\right) &\rTo \VV^*(-h),\\
    H^4_{\et}\left(Y_{\Gh,\QQbar}, \hat{\cV}^\vee(3-q-t_2)\right) &\rTo \VV^*(-h),\\
    H^5_{\et}\left(Y_{\Gt,\QQbar}, \tilde{\cV}^\vee(3-q-t_1-t_2)\right) &\rTo \VV^*(-h).
   \end{align*}
   respectively.


 \subsection{Galois-cohomological periods}

  \begin{definition}
   Let $z$ denote the unique map
   \begin{align*}
    \mathscr{W}(\Pif')_L \times \cS_{(0,0)}(\Af^2\times\Af^2) &\to H^1(\QQ, \VV^*(-h))[2-h] & \text{(case A)},\\
    \mathscr{W}(\hat{\Pi}_{\f}')_L \times \cS_{(0)}(\Af^2) &\to H^1(\QQ, \VV^*(-h))[2-h] & \text{(case B)},\\
    \mathscr{W}(\tilde{\Pi}_{\f}')_L  &\to H^1(\QQ, \VV^*(-h))[2-h] & \text{(case C)},
   \end{align*}
   such that, respectively,
   \begin{align*}
    \big\langle\!\!\big\langle \LE_{\et}(\Phi \otimes \xi), w \big\rangle\!\!\big\rangle_{G}& = z(\xi \cdot w, \Phi),\\
    \big\langle\!\!\big\langle \widehat{\LE}_{\et}(\Phi \otimes \xi),  \hat{w} \big\rangle\!\!\big\rangle_{\Gh}& = z(\xi \cdot \hat{w}, \Phi),\\
    \big\langle\!\!\big\langle \widetilde{\LE}_{\et}(\xi),  \tilde{w} \big\rangle\!\!\big\rangle_{\Gt} &= z(\xi \cdot \tilde{w}).
   \end{align*}
  \end{definition}

  \begin{notation}
   Let $\mathscr{S}$ and $\mathscr{T}$ (for ``test data'') denote the following spaces:
   \[
    \mathscr{S} =
    \begin{cases}
     \cS_{(0)}(\Af^2, \hat\chi_{\Sigma_1})_L[-t_1] \boxtimes \cS_{(0)}(\Af^2, \hat\chi_{\Sigma_2})_L[-t_2] & \text{(case A)} \\
     \cS_{(0)}(\Af^2, \hat\chi_{\Sigma_1})_L[-t_1] \boxtimes \sW(\Sigma_{2, \f}')_L & \text{(case B)} \\
     \phantom{\cS_{(0)}(\Af^2, \hat\chi_{\Sigma_1})_L[-t_1]}\mathllap{\sW(\Sigma_{1, \f}')_L} \boxtimes \sW(\Sigma_{2, \f}')_L & \text{(case C)} \\
    \end{cases},
   \]
   and
   \[ \mathscr{T} = \sW(\Pif') \boxtimes \mathscr{S}, \]
   considered as a $\Gt(\Af)$-module in the evident fashion. (Here, $\cS_{(0)}(\Af^2, \hat\chi)_L$ denotes the $\hat\chi$-eigenspace for the action of $\widehat{\ZZ}^\times$.)

  \end{notation}

  One checks that $z(\dots)$ is an $H(\Af)$-equivariant map $\mathscr{T} \to H^1(\QQ, \VV^*(-h))[2-h]$, where $H(\Af)$ acts on $\mathscr{T}$ via $\tilde{\iota}$, and trivially on $H^1(\QQ, \VV^*(-h))$.

  \begin{remark}[Multiplicity one]
   In case (C), the main result of \cite{waldspurger12b} shows that we have
   \[ \dim \Hom_{H(\Af)}\left(\mathscr{T}, L[2-h]\right) \le 1,\]
   with equality iff the local root numbers at all finite places are $+1$. So the image of $z(\dots)$ is contained in a subspace of $H^1(\QQ, \VV^*(-h))$ of dimension $\le 1$.

   The analogous statements in cases (A) and (B) are more delicate, since one or both of the $\Sigma_i$ is replaced with a principal-series representation (a quotient of a space of Schwartz functions) which is no longer irreducible if $t_i = 0$. In this reducible case the multiplicity-one result is no longer automatic, due to ``exceptional pole'' phenomena for local $L$-factors. This is explored in more detail in \cite{loeffler-zeta1}. In any case, these multiplicity-one results do not seem to be needed for the proof of our main theorem, so we shall not pursue the matter further in the present paper.
  \end{remark}

  Unravelling the notation, we can write the map $z(\dots)$ explicitly in terms of the pushforward maps of \cref{lem:goodsubrep}. We omit case A here since this is given in \cite[\S 6.5]{LZ20}.

  \begin{proposition}\
   \label{prop:explicitZ}
   \begin{itemize}
    \item In case (B), suppose that $\hat{w} \in \mathscr{W}(\hat{\Pi}_{\f}')$. Then for any open compact $U \subset \Gh(\Af)$ such that $U$ fixes $\hat{w}$ and $V = U \cap H(\Af)$ fixes $\Phi$, we have
    \[ z(\hat{w}, \Phi) = \operatorname{vol}(V) \cdot \Big\langle \hat\iota_{U, \star}\left(\Eis^{t_1}_{\et, \Phi}\boxtimes\mathbf{1}\right),\, \hat{w}\Big\rangle_{\hat{Y}(U)}. \]
    Here $\hat\iota_{U, *}$ denotes pushforward along $Y_H(V) \to Y_{\Gh}(U)$ and $\Eis^{t_1}_{\et, \Phi}$ denotes the \'etale realisation of the motivic Eisenstein class $ \Eis^{t_1}_{\mot, \Phi}$.
    \item In case (C), suppose that $\tilde{w} \in \mathscr{W}(\tilde{\Pi}_{\f}')$. Then for any open compact $U \subset \Gt(\Af)$ such that $U$ fixes $\hat{w}$, letting $V = U \cap H(\Af)$, we have
     \[ z(\tilde{w}) = \operatorname{vol}(V) \cdot \Big\langle \tilde\iota_{U, \star}\left(\mathbf{1}\right),\, \tilde{w}\Big\rangle_{\tilde{Y}(U)}. \]
     Here $\tilde\iota_{U, *}$ denotes pushforward along $Y_H(V) \to Y_{\Gt}(U)$.
   \end{itemize}
  \end{proposition}


 \subsection{Exponential maps and regulators}

  \begin{proposition}
   In cases (A) and (C) we have
   \[
    H^1_{\mathrm{e}}(\Qp, \VV^*(-h)) = H^1_{\mathrm{f}}(\Qp, \VV^*(-h))= H^1_{\mathrm{g}}(\Qp, \VV^*(-h))
   \]
   where $H^1_{\mathrm{e}}$, $H^1_{\mathrm{f}}$ and $H^1_{\mathrm{g}}$ are the Bloch--Kato subspaces. (Cf.~\cite[Lemma 6.7.2]{LZ20}).

   In case (B), we always have $H^1_{\mathrm{e}}(\Qp, \VV^*(-h)) = H^1_{\mathrm{f}}(\Qp, \VV^*(-h))$. We also have $H^1_{\mathrm{f}}(\Qp, \VV^*(-h)) = H^1_{\mathrm{g}}(\Qp, \VV^*(-h))$ \textbf{unless} $(t_1, t_2) = (0, r_1 - r_2)$ and one of the pairwise products $\{ \alpha \fa_2, \dots, \delta \fb_2\}$ is equal to $p^{r_1 + 2}$.
  \end{proposition}

  \begin{proof}
   It is well known that if $V$ is a crystalline representation of $\Gal(\QQbar_p / \Qp)$ and neither $1$ nor $p^{-1}$ is an eigenvalue of $\varphi$ on $\Dcris(V)$, then $H^1_{\mathrm{e}}(V) = H^1_{\f}(V) = H^1_{\mathrm{g}}(V)$.

   Since $\Pi$ and the $\Sigma_i$ are tempered, we can compute the absolute values of the eigenvalues of $\varphi$ in terms of the $r_i$ and $t_i$. We find that the common absolute value is $p^{-(t_1 + t_2 + 3)/2}$ in case (A), $p^{-(t_1 + 2)/2}$ in case (B), and $p^{-1/2}$ in case (C). So $1$ is never an eigenvalue; and $p^{-1}$ can only be an eigenvalue if we are in case (B) and $t_1 = 0$, which implies $t_2 = r_1 - r_2$ and hence $h = r_1 + 1$, and one of the numbers $\{ \tfrac{p^h}{\alpha\fa_2}, \dots, \tfrac{p^h}{\delta\fb_2}\}$ is equal to $p^{-1}$.
  \end{proof}

  \begin{remark}
   The exceptional case corresponds to a situation where the $L$-value corresponding to our motivic class is a near-central value, and there is an ``exceptional zero'' coming from the Euler factor at $p$. Note that the relation $t_2 = r_1 - r_2$ implies that the $L$-function of $\Pi \times \Sigma_2$ has no critical values.
  \end{remark}

  \begin{lemma}
   Suppose $\Pi$ is Klingen-ordinary at $p$ and we are in Case B, with $(t_1, t_2) = (0, r_1 - r_2)$. If one of the pairwise products $\{ \alpha \fa_2, \dots, \delta \fb_2\}$ is equal to $p^{r_1 + 2}$, then $\Pi$ must in fact be Borel-ordinary at $p$, $\Sigma_2$ is ordinary at $p$, and the ``bad'' eigenvalue is either $\beta \fb_2$ or $\gamma \fa_2$ (or possibly both) where $\alpha$ and $\fa_2$ are the unit eigenvalues.
  \end{lemma}

  \begin{proof}
   This follows by elementary manipulations from the inequalities for Hecke parameters of $\Pi$ and $\Sigma_2$.
  \end{proof}

  \begin{assumption}
   We assume henceforth that $\Pi$ is Klingen-ordinary at $p$, and that we are not in the exceptional case-B situation described above.
  \end{assumption}

  Since the localisations at $p$ of the classes $z(\dots)$ lie in $H^1_{\mathrm{g}}$ (by \cite[Theorem B]{nekovarniziol}), they are also in $H^1_{\mathrm{e}}$. Letting ``$\log$'' denote the inverse of the Bloch--Kato exponential, for any $\tau \in \mathscr{T}$ we may define
  \begin{align*}
   \log z(\tau)\in \frac{\DdR(\VV^*)}{\Fil^{-h} \DdR(\VV^*)}&=\left(\Fil^{(1+h)}\DdR(\VV)\right)^*.
  \end{align*}

  \begin{definition}
   Let $\nu$ be a basis of the 1-dimensional $L$-vector space $\frac{\Fil^1\DdR(V_p(\Pi))}{\Fil^{r_2+2}\DdR(V_p(\Pi))}$, and $\nu_{\dR}$ its unique lifting to $\Fil^1\DdR(V_p(\Pi)) \cap \Dcris(V_p(\Pi))^{(\varphi - \alpha)(\varphi-\beta) = 0}$, as in \cite[\S 6.7]{LZ20}.
  \end{definition}

  One checks that
  \[ \tilde\nu_{\dR} \coloneqq \nu_{\dR} \otimes\nu^{(1)}_{\dR}\otimes \nu^{(2)}_{\dR} \in \Fil^{(1+h)} \DdR(\VV), \]
  so this can be meaningfully paired with $z(\tau)$.

  \begin{definition}
   \label{def:regulators}
   For $\tau \in \mathscr{T}$, we set
   \[ \operatorname{Reg}_\nu(\tau) = \left\langle \log z(\tau), \mu_{\dR}\right\rangle_{\Dcris(\VV)} \in L.\]
  \end{definition}

  This defines an $H(\Af)$-equivariant map $\mathscr{T} \to L[2-h]$.

 \subsection{Multiplicity one at $p$}

  \begin{definition}
   Write $\mathscr{T} = \mathscr{T}_p \otimes \mathscr{T}^p$ as a tensor product of test data at $p$ and away from $p$; and let $\tau_{p,\sph} \in \mathscr{T}_p$ denote the ``normalised spherical datum'' at $p$, given as follows:
   \begin{itemize}
    \item in case (A), $\tau_{p,\sph} = w_{0, \sph} \otimes \Phi_{\sph} \otimes \Phi_{\sph}$, where $w_{0, \sph}$ is the spherical Whittaker function of $\Pi'_p$ with $w_{0, \sph}(1) = 1$, and $\Phi_{\sph} = \ch(\Zp^2)$;
    \item in case (B), $\tau_{p,\sph} = w_{0, \sph} \otimes \Phi_{\sph} \otimes w_{2, \sph}$, where $w_{2, \sph}$ is the spherical Whittaker function of $\Sigma_{2, p}'$;
    \item in case (C), $\tau_{p,\sph} = w_{0, \sph} \otimes  w_{2, \sph} \otimes w_{2, \sph}$.
   \end{itemize}
  \end{definition}

  \begin{proposition}
   We have $\dim \Hom_{H(\Qp)}(\mathscr{T}_p, L[2-h]) = 1$, and this space has a unique basis vector $\tilde{Z}_p$ such that $\tilde{Z}_p(\tau_{p, \sph}) = 1$.
  \end{proposition}

  \begin{proof}
   In case (A), this is \cite[Proposition 6.9.1]{LZ20}, which follows from a delicate analysis of local integrals carried out in \S 19 of \emph{op.cit.}. The same arguments are valid in cases (B) and (C); in case (B), we need to assume that $L$-factor $L(\Pi_p \otimes \Sigma_{2, p}, s)$ does not have a pole at $s = -\tfrac{t_1}{2}$, but this turns out to be equivalent to the assumption we have already made, that we are not in the ``exceptional case B'' situation described above.
  \end{proof}

  Consequently, there are $H(\Af^p)$-equivariant maps
  \begin{align*}
   z^{p}:  \mathscr{T}^p &\to H^1(\QQ, \VV(-h))[2-h], & \operatorname{Reg}_\nu^p: \mathscr{T}^p &\to L[2-h],
  \end{align*}
  such that we have
  \begin{align*}
   z(\tau_p \tau^p) &= \tilde{Z}_p(\tau_p)z^{p}(\tau^p), & \operatorname{Reg}_\nu(\tau_p\tau^p) &= \tilde{Z}_p(\tau_p)\operatorname{Reg}^{p}_\nu(\tau^p).
  \end{align*}
  Our goal is to give a formula for $\operatorname{Reg}_\nu^p(\tau^p)$ in terms of a cup-product in coherent cohomology, which we can then identify with a value of a suitable $p$-adic $L$-function.

 \subsection{Klingen-type data at $p$}

  We now specify the local data we want to use at $p$. Recall that we have assumed $\Pi$ to be Klingen-ordinary at $p$. There is thus a unique ordinary eigenspace at Klingen parahoric level $\Kl(p)$ for the operator $U_{2, \Kl}'$.

  \begin{definition} \
   \begin{itemize}
    \item Let $w'_{\Kl} \in \sW(\Pi_p')_E$ denote the normalised $U_{2, \Kl}'$-eigenvector at level $\Kl(p)$ defined in \cite[\S 20.1]{LZ20}.
    \item For $i = 2$ in case B, and $i = 1, 2$ in case C, we choose arbitrarily an ordering of the Hecke parameters  $(\fa_i, \fb_i)$, and we let $w_i^{\fa_i} \in \sW(\Sigma_{i,p}')$ be the $(U = \fa_i)$-eigenvector at level $\Iw_{\GL_2}(p)$.
    \item Let $\Phi_{\crit} \in \cS(\Qp^2, \QQ)$ be the unique Schwartz function such that $\Phi_{i, p}' = \ch(\Zp \times \Zp^\times)$, where $(\,)'$ denotes Fourier transform in the second variable only (cf.~\cite[\S 15.2]{LZ20}).
   \end{itemize}
   We define
   \[ \tau_p^{\Kl} =
   \begin{cases}
    \gamma_p w'_{\Kl} \boxtimes \Phi_{\crit} \boxtimes \Phi_{\crit} & \text{case A} \\
    \gamma_p w'_{\Kl} \boxtimes \Phi_{\crit} \boxtimes w_2^{\fa_2} & \text{case B} \\
    \gamma_p  w'_{\Kl} \boxtimes \phantom{\Phi_{\crit}} \mathllap{w_1^{\fa_1}} \boxtimes w_2^{\fa_2} & \text{case C}, \\
   \end{cases}\]
   where $\gamma_p \in \GSp_4(\Zp) / \Kl(p)$ has first column $(1,1,0,0)^T \bmod p$.
  \end{definition}

  \begin{proposition}
   \label{prop:klingendata}
   We have
   \[
    \tilde{Z}_p(\tau_p^{\Kl}) =
    \prod_{ \xi \in \{
         \fa_1 \fa_2,\,
         \fa_1 \fb_2,\,
         \fb_1 \fa_2\}} \left(1 - \tfrac{p^h}{\alpha\xi}\right)\left(1 - \tfrac{p^h}{\beta\xi}\right).
   \]
   In particular, if we are not in the ``exceptional type B'' case, then $\widetilde{Z}_p(\tau_p^{\Kl}) \ne 0$.
  \end{proposition}

  \begin{proof}
   This follows from a local zeta-integral computation which is carried out in \cite{loeffler-zeta2}.
  \end{proof}

\section{Coherent cup products}

 \subsection{Setup}
 \label{sect:periods}

  In this section, for each of the cases (A), (B), (C) we shall define a linear map
  \[ \operatorname{Per}^p_\nu: \mathscr{T}^p \to L. \]
  This will be defined using the Zariski cohomology of coherent sheaves on a $p$-adic integral model of the $\GSp_4$ Shimura variety (of Klingen parahoric level at $p$), in contrast to the \'etale cohomology used to define the regulator $\operatorname{Reg}^p_\mu$. However, it is important to note that both classes depend on the same input data, so it is meaningful to compare the two.

  We shall recall below that the values of $\operatorname{Per}^p_\nu$ can be interpreted as values (outside the interpolation range) of a suitable $p$-adic $L$-function. On the other hand, the main result of this paper will be to relate $\operatorname{Reg}^p_\nu$ to $\operatorname{Per}^p_\nu$.

  \begin{assumption}
   Henceforth we suppose $r_2 \ge 1$.
  \end{assumption}

 \subsection{Coherent classes from $\GL_2$ cuspforms}

  As noted above, for $i = 2$ in case (B), and $i = 1,2$ in case (C), there are canonical homomorphisms $\nu_i: \sW(\Sif')_E \to H^0(X_{\GL_2, E}, \omega^{t_i + 2})$.

  In particular, given any choice of vector $w_{p} \in \sW(\Sigma_{i, p}')$, we can form a map $\sW(\Sif^{\prime (p)})_E \to H^0(X_{\GL_2, E}, \omega^{t_i + 2})$ via $w^p \mapsto \nu_i(w^p w_{p})$. The relevant choices of $w_p$ will be:
  \begin{itemize}
   \item the spherical vector $w^{\sph}$, normalised by $w_{i, p}^{\sph}(1) = 1$;
   \item the vectors $w^{\fa_i}$ and $w^{\fb_i}$, which are the unique vectors invariant under $\{ \stbt{*}{*}{}{\star} \bmod p\}$ which are $U_p$-eigenvectors with eigenvalues $\fa_i$ and $\fb_i$ respectively, and satisfy $w^{\fa_i}(1) = w^{\fb_i}(1) = 1$;
   \item the ``$p$-depleted vector'' $w^{\mathrm{dep}}$, which is characterised by $w^{\mathrm{dep}}\left(\stbt{x}{0}{0}{1}\right)= \ch_{\Zp^\times}(x)$.
  \end{itemize}

  We denote the corresponding homomorphisms from $\sW(\Sif^{\prime (p)})_E$ by $(\nu_i^{\sph}$, $\nu_i^{\fa_i}$, $\nu_i^{\dep})$. After base-extending to $L$, we can view $\nu_i^{\dep}$ as taking values in the kernel of the $U_p$-operator acting on the space $\mathbf{S}_{t_i+2}(L)$ of $p$-adic modular forms. On this space, the Serre differential operator $\theta$ (acting as $q \tfrac{\mathrm{d}}{\mathrm{d}q}$ on $q$-expansions) is invertible; so we can make sense of $\theta^{-n}(\nu_i^{\dep})$, for any $n \in \ZZ$, as a map
  \[ \sW\left(\Sif^{\prime (p)}\right)_L \to \mathbf{S}_{t_i + 2 - 2n}(L). \]

 \subsection{Coherent classes from Eisenstein series}

  We now recall some facts from \cite[\S 15]{LZ20} on Eisenstein series. If $k \ge 1$,  and $\Phi \in \cS_{(0)}(\Af^2)$, we defined in \cite{LPSZ1} an Eisenstein series $E^{(k, \Phi)}(s)$ depending on a complex parameter $s$, which is a $C^\infty$ real-analytic section of $\omega^k$ on $Y_{\GL_2}$. We let $E^k_{\Phi} = E^{(k, \Phi)}(\tfrac{k}{2})$ and $F^k_{\Phi} = E^{(k, \Phi)}(1-\tfrac{k}{2})$, which are holomorphic modular forms; if $\Phi$ is $E$-valued they are defined over $E$ as sections of $\omega^k$.

  \begin{proposition}
   The de Rham realisation of the Eisenstein class $\Eis^t_{\mot, \Phi}$ is given by $F^{(t+2)}_{\Phi}$.
  \end{proposition}

  We shall take $\Phi = \Phi^p \Phi_p$, where $\Phi^p$ is an arbitrary Schwartz function away from $p$, and $\Phi_p$ is one of the following. Here $\Phi'_p$ denotes the Fourier transform in the second variable only, as in \cite[\S 15.2]{LZ20}.
  \begin{itemize}
  \item the spherical Schwartz function $\Phi_{\sph} = \ch(\Zp \times \Zp)$;
  \item the ``critical-slope'' Schwartz function $\Phi_{\crit}$ such that $\Phi'_{\crit} = \ch(\Zp \times \Zp^\times)$;
  \item the ``depleted'' Schwartz function $\Phi_{\dep}$ such that $\Phi'_{\dep} = \ch(\Zp^\times \times \Zp^\times)$.
  \end{itemize}

  For $i = 1, 2$ in case (A), and $i = 1$ in case (B), we write $\mu_i^{\dep}$ for the map
  \[ \cS_{(0)}\left( (\Af^{(p)})^2, \widehat{\chi}_{\Sigma_i}\right)[-t_i] \to \mathbf{S}_{t_i + 2}(L), \qquad \Phi^p \mapsto F^{(t_i+2)}_{\Phi^p \Phi_{\dep}}.\]
  We define $\mu_i^{\fa_i}$ similarly, using $\Phi_{\crit}$ in place of $\Phi_{\dep}$. The values of the maps $\mu_i^{\dep}$ and $\mu_i^{\crit}$ are modular forms which vanish at all cusps in the multiplicative locus of the modular curve, and hence are cuspidal as $p$-adic modular forms.

  \begin{definition}
   In all three cases, we define a map
   \[ \cG: \mathscr{S}^p \to \mathbf{S}_{r_1 - r_2 - t_2}(L) \boxtimes \mathbf{S}_{t_2 + 2}(L), \qquad
   \cG = \theta^{-(r_2 - q + 1)}\left(\mu_1^{\dep}\right) \boxtimes \mu_2^{\dep}. \]
  \end{definition}

  Note that the values of $\cG$ are $p$-adic modular forms which are never classical (unless they are 0).

 \subsection{Coherent classes from $\Pi$}

  Our final ingredient is a map
  \[ \nu_{\Kl}: \cW(\Pif^{\prime (p)}) \to H^2(X_{G, \Kl}(p), \cN^1).\]
  Write $\eta_{\Kl}$ for a generic element in the image of this map, and $\mu_{\dR}$ for the canonical lift to $\Fil^1\DdR(V_p(\Pi))$.
  We can regard $\nu$ as a map from $\cW(\Pif')$ to this space, so as above, it suffices to choose a Whittaker vector at $p$. We choose this to be the normalised generator of the $U_{2, \Kl}'$-ordinary eigenspace at Klingen parahoric level defined in \cite[\S 20.1]{LZ20}.

  We can consider the values of $\nu_{\Kl}$ as linear functionals on the $U_{2, \Kl}$-ordinary part of $H^1(X_{G, \Kl}(p), \cN^2(-D))$. By the classicity theorem proved in \cite[\S 3]{LPSZ1}, we can regard $\nu_{\Kl}$ as a linear functional on the cohomology of the multiplicative locus, $H^1(X_{G, \Kl}^{\ge 1}, \cN^2(-D))$, since these two spaces have the same ordinary part.

  \begin{definition}
   We define
   \( \operatorname{Per}_{\nu}: \mathscr{T}^p \to L\)
   by
   \[ \operatorname{Per}_{\nu} = \Big\langle \nu_{\Kl}, \iota_{1, \star}(\cG |_{X_H})\Big\rangle_{X_{G, \Kl}^{\ge 1}}, \]
   where
   \[ \iota_{1, \star}: H^0\left( X_{H, \Delta}^{\ord}(p), \omega^{(r_1 - r_2 - t_2, t_2 + 2)}(-D)\right) \longrightarrow
    H^1\left(X_{G, \Kl}^{\ge 1}(p), \cN^2(-D)\right) \]
   is the twisted pushforward map considered in \cite[\S 4]{LPSZ1}.
  \end{definition}

 \subsection{Statement of the main theorem}

  \begin{theorem}
   \label{thm:main}
   We have
   \[
    \operatorname{Reg}_\nu(\tau^p \tau_p^{\Kl})=
    \frac{(-2)^q (-1)^{r_2 - q  + 1}(r_2 - q)!}
    {\left(1 - \tfrac{p^h}{\alpha \fb_1 \fb_2}\right)
     \left(1 - \tfrac{p^h}{\beta \fb_1 \fb_2}\right)}
    \operatorname{Per}_\nu^p(\tau^p),
   \]
   for all choices of $\tau^p \in \mathscr{T}^p$.
  \end{theorem}

  Unravelling the notation a little, we can write this in the following form, which is what we shall actually prove. For $K^p$ an open compact subgroup of $G(\Af^p)$, write $K_1^p$ and $K_2^p$ for the first and second projections of $K_H^p = K^p \cap H$; and let
  \[ \hat{K}^p = K^p \times K_2^p,\qquad \tilde{K}^p = K^p \times K_1^p \times K_2^p.\]
  Meanwhile, let $\Kl(p)$ denote the Klingen parahoric,
  \[\hat{\Kl}(p) = \Kl(p) \times \Iw_{\GL_2}(p),\qquad \tilde{\Kl}(p) = \Kl(p) \times \Iw_{\GL_2}(p)\times\Iw_{\GL_2}(p); \]
  and let $K_{H, \Delta}(p) = H(\Zp) \cap \gamma_p \Kl(p) \gamma_p^{-1}$, described explicitly in \cite[\S 4.1]{LPSZ1}. We write $Y_{\Kl}(p), \widehat{Y}_{\Kl}(p), \widetilde{Y}_{\Kl}(p)$ for the Shimura varieties of level $K^p\Kl(p), \hat{K}^p \hat{\Kl}(p)$ etc, and similarly $Y_{H, \Delta}(p)$. We write $X_{\dots}$ for their toroidal compactifications.

  \begin{proposition}
   \label{prop:weprove}
   In case (C), the above formula is equivalent to the following statement: for all prime-to-$p$ levels $K^p \subset G(\Af^p)$, all $\eta \in e_{\Kl}' H^2(\Pif)^{K^p \Kl(p)}$, and all $\mu_i \in H^0(\Sif')^{K^p_i\Iw_{\GL_2}(p)}$ which are $U_p = \fa_i$ eigenvectors, we have
   \begin{multline*}
    \Big\langle \log \pr_{\tilde{\Pi}}\tilde{\iota}_{1, \star}(\mathbf{1}), \eta_{\dR} \boxtimes \mu_{1, \dR} \boxtimes \mu_{2, \dR}\Big\rangle_{\widetilde{Y}_{\Kl}(p)} \\
    = \frac{(-2)^q (-1)^{r_2 - q  + 1}(r_2 - q)!}
    {\left(1 - \tfrac{p^h}{\alpha \fb_1 \fb_2}\right)
    \left(1 - \tfrac{p^h}{\beta \fb_1 \fb_2}\right)}
    \Big\langle \eta, \iota_{1, \star}\left( \theta^{-(1+a)}(\mu_1^{[p]}) \boxtimes \mu_2^{[p]} \right)\Big\rangle_{X_{G, \Kl}^{\ge 1}(p)}.
    \end{multline*}
   Here $\mu_i^{[p]}$ is the $p$-depletion of $\mu_i^{\fa_i}$.
   There are analogous formulae in cases (B) and (A) in which the left-hand side is replaced with
   \[ \Big\langle \log \pr_{\hat{\Pi}}\hat{\iota}_{1, \star}(\Eis^{t_1}_{\et, \Phi_1^p\Phi_{\crit}} \boxtimes \mathbf{1}), \eta_{\dR} \boxtimes \mu_{2, \dR}\Big\rangle_{\widehat{Y}_{\Kl}(p)}, \]
   or
   \[ \Big\langle \log \pr_{\Pi} \iota_{1, \star}(\Eis^{t_1}_{\et, \Phi_1^p\Phi_{\crit}} \boxtimes \Eis^{t_2}_{\et, \Phi_2^p\Phi_{\crit}}), \eta_{\dR}\Big\rangle_{Y_{\Kl}(p)}, \]
   for any choices of $(\eta, \Phi_1^p, \mu_2)$, resp. $(\eta, \Phi_1^p, \Phi_2^p)$.
  \end{proposition}

  The equivalence of this and \cref{thm:main} follows from \cref{prop:explicitZ}. It is the formula of \cref{prop:weprove} that we shall actually prove.

  \begin{remark}
   Note that the class $\eta_{\dR}$ is annihilated by $(1 - \tfrac{\varphi}{\alpha})(1 - \tfrac{\varphi}{\beta})$, and $\mu_{i, \dR}$ is annihilated by $\left(1 - \tfrac{\varphi}{\fa_i}\right)\left(1 - \tfrac{\varphi}{\fb_i}\right)$. Here $\varphi$ denotes the ``abstract'' Frobenius of rigid cohomology. Over the ordinary locus, this Frobenius element has a canonical lifting given by a suitable Hecke operator, and we shall see that the restriction of $\mu_{i, \dR}$ to this locus is in fact killed by the simpler polynomial $\left(1 - \tfrac{\varphi}{\fb_i}\right)$. (A little care is required here in the the non-cuspidal cases, since in these cases $\mu_{i, \dR}$ itself is exact on the ordinary locus; but it is not ``cuspidally exact'' -- it does not have a cuspidal antiderivative -- and hence represents a non-trivial class in \textit{compactly-supported} cohomology, which is annihilated by $1 - \tfrac{\varphi}{\fb_i}$.)

   So $\left(1 - \tfrac{\varphi}{\alpha \fb_1 \fb_2}\right)\left(1 - \tfrac{\varphi}{\beta \fb_1 \fb_2}\right)$ annihilates $\eta_{\dR} \boxtimes \mu_{1, \dR} \boxtimes \mu_{2, \dR}$ on this region, which explains the appearance of the Euler factor in the above formulae.
  \end{remark}

 \section{The relation to $p$-adic $L$-functions}

  In this section we recall how the coherent period $\operatorname{Per}_{\nu}(\tau^p)$ is related to values of $p$-adic $L$-functions, summarising results proved in more detail in \cite{LPSZ1} for cases (A), (B), and in \cite{LZvista} for case (C).

  \subsection{Cases A and B}

  In case A, and in a subcase of case B (when $t_2 \le r_1 - r_2 - 1$), we can deform the Eisenstein representations $p$-adically, while keeping the cuspidal automorphic representations fixed; this simplifes the statements considerably. We shall also assume that the field $E$ contains a root of unity of order $N_2$, where $N_2$ is the conductor of  $\chi_{\Sigma_2}$; this allows us to sidestep some rationality issues involving Gauss sums.

  \begin{theorem} \
   \begin{enumerate}
    \item In case (A), there exists an element $\cL_{p, \nu}(\Pi; \chi_{\Sigma_1}, \chi_{\Sigma_2}) \in \Lambda_L(\Zp^\times \times \Zp^\times)$ with the following property: for $a_1 + \rho_1, a_2 + \rho_2$ locally algebraic characters with $0 \le a_i \le r_1 - r_2$ and $(-1)^{a_1 + a_2}\rho_1\rho_2(-1) = -\chi_{\Sigma_2}(-1)$, we have
    \[ \cL_{p, \nu}(\Pi; \chi_{\Sigma_1}, \chi_{\Sigma_2})(a_1 + \rho_1, a_2 + \rho_2) = (\star) \cdot \frac{L(\Pi \times \rho_1^{-1}, \tfrac{1-r_1+r_2}{2} + a_1) L(\Pi \times \chi_{\Sigma_2} \rho_2^{-1}, \tfrac{1-r_1+r_2}{2} + a_2)}{\Omega_{\Pi}(\nu)}. \]
    \item In case (B), suppose $t_2 \le r_1 - r_2 - 1$. Then there exists an element $\cL_{p, \nu}(\Pi \times \Sigma_2) \in \Lambda_L(\Zp^\times)$ with the following property: for $a + \rho$ a locally algebraic character with $0 \le a \le r_1 - r_2 - 1 - t_2$, we have
     \[ \cL_{p, \nu}(\Pi \times \Sigma_2)(a + \rho) = (\star) \cdot \frac{L(\Pi \times \Sigma, \tfrac{2-r_1+r_2+t_2}{2} + a)}{\Omega_{\Pi}(\nu)}. \]
   \end{enumerate}
  \end{theorem}

  The following formula for the coherent period $\operatorname{Per}_{\nu}$ drops out naturally from the construction:

  \begin{proposition}
   In case (A) we have
   \[ \operatorname{Per}_{\nu}(\tau^p) = \cL_{p, \nu}(\Pi; \chi_{\Sigma_1}, \chi_{\Sigma_2})(-1-r_2 + q, r) \cdot \prod_{\ell \ne p} \widetilde{Z}_\ell(\tau_\ell) \]
   where $\widetilde{Z}_\ell$ is a non-zero linear functional on $\mathscr{T}_\ell$, which is 1 on the spherical data if $\ell$ is an unramified prime. In case (B), we have the similar formula
   \[ \operatorname{Per}_{\nu}(\tau^p) = \cL_{p, \nu}(\Pi \times \Sigma_2)(-1-r_2 + q) \cdot \prod_{\ell \ne p} \widetilde{Z}_\ell(\tau_\ell). \]
  \end{proposition}

  Note that $-1-r_2 + q < 0$, so these values always lie \emph{outside} the range of interpolation of the $p$-adic $L$-function. To include case (C), and the remaining subcases of (B), in our discussion we need to allow the $\GL_2$ cusp forms to vary in $p$-adic families.

 \subsection{Hida families and case (C)}

  For case C, we use the following setup. The theory of local newvectors shows that $\mathscr{T}^p$ has a canonical non-zero vector $\tau^p_{\mathrm{new}}$ such that $\mathscr{T}^p = \Gt(\Af^p) \cdot \tau^p_{\mathrm{new}}$. Fix an element
  \[ \gamma^{(p)}=(\gamma_0^{(p)},\gamma_1^{(p)},\gamma_2^{(p)})\in H(\Af^{p}) \backslash \Gt(\Af^p).\]

  \begin{theorem}
   \label{thm:padicLfcnA}
   (c.f. \cite[Thm. 4.4.1]{LZvista}) Let $\underline{\cG}=(\cG_{1},\cG_{2},\cc)$ be a self-dual twist of a pair of Hida families, with $\chi_\Pi \cdot \chi_{\cG_1}\cdot \chi_{\cG_2}=1$, with coefficients in a finite extension $\mathbb{I}$ of $\Lambda_{\cO}(\ZZ_p^\times\times\ZZ_p^\times)$.
   There exists a $p$-adic $L$-function
   \[
    \cL^{(f)}_{p, \nu, \gamma^{(p)}}(\Pi\times \ucG)(\kappa_1,\kappa_2)
    \in \mathbb{I}\otimes_{\cO}\bar{L},
   \]
   which interpolates the Gan--Gross--Prasad period on $\Pi\times\Sigma(\cG_{1,\kappa_1})\times\Sigma(\cG_{2,\kappa_2})$ in the region
   \[ \mathfrak{X}_{(f)}=\{ (\kappa_1,\kappa_2): \kappa_i\geq 1,\, \kappa_1+\kappa_2 \leq r_1-r_2+2, \kappa_1 + \kappa_2 = r_1 + r_2 \bmod 2\}\subset \mathfrak{X}_{\cG_1}\times \mathfrak{X}_{\cG_2}. \]
  \end{theorem}

  \begin{proposition}
   \label{prop:periodformulaC}
   For $(t_1, t_2)$ such that $(r_1, r_2, t_1, t_2)$ satisfies \cref{prop:branchingG}, we have the formula
   \[ \operatorname{Per}_\nu(\gamma^{(p)} \cdot \tau^p_{\mathrm{new}}) = \cL^{(f)}_{p,\nu,\gamma^{(p)}}(\Pi\times \ucG)(t_1 + 2, t_2 + 2).\]
  \end{proposition}

  \begin{remark}
   The set of values $(\kappa_1, \kappa_2) = (t_1+2, t_2+2)$ to which \cref{prop:periodformulaC} applies corresponds to region (e) in \cite[Figure 2]{LZvista}; in particular, it has empty intersection with the region (f) in which $\cL^{(f)}_{p,\nu,\gamma^{(p)}}(\Pi\times \ucG)$ interpolates the automorphic period.

   We can likewise extend the $p$-adic $L$-function of case (B) to allow $\Sigma_2$ to vary in a Hida family (with $\gamma^{(p)}$ now denoting a triple $(\gamma_0^{(p)}, \Phi_1^{(p)}, \gamma_2^{(p)})$).
  \end{remark}



\section{Syntomic and finite-polynomial cohomology}

 \subsection{Nekovar--Niziol fp-cohomology}\label{sect:NNfp}

  We briefly summarize the results from \cite[\S\S 7.2 - 7.5]{LZ20}.

  Let $X$ be any $\Qp$-variety, and let $n \in \ZZ$. One can define groups  $R\Gamma_{\NNfp}(X, n;P)$ and $R\Gamma_{\NNfp,c}(X, n;P)$ for any polynomial $P(t)\in\Zp[t]$ with constant coefficient $1$, with the case $P(t)=1-t$ recovering the syntomic theory of \cite{nekovarniziol}. In this case, we denote the cohomology groups by $R\Gamma_{\NNsyn}(X, n)$.

If $X$ is smooth of pure dimension $d$, then we have a canonical trace map
\begin{equation}\label{eq:NNtrace}
\tr_{\NNfp, X, P}:  H^{2d+1}_{\NNfp, c}(X, d+1, P) \to H^1_{\st, P}(\Qp, \Qp^{\mathrm{nr}}(1)) \to \Qp, \
\end{equation}
for any polynomial $P$ such that $P(1/p) \ne 0$.

   If $\mathfrak{G}$ is a reductive group and $U\subset \mathfrak{G}(\AA_f)$ open compact, write $X=Y_{\mathfrak{G}}(U)$ for the Shimura variety attached to $\mathfrak{G}$ of level $U$. We can define cohomology with coefficients in algebraic representations $V$. The cup product gives rise to a pairing
   \begin{equation}\label{eq:NNfppairing}
    \langle \quad,\quad\rangle_{\NNfp,X,P}: H^{i}_{\NNsyn}(X, \cV,n)\times H^{2d+1-i}_{\NNfp, c}(X, \cV^\vee,d+1-n, P)\rTo \Qp
   \end{equation}
   for any $n$ and any $P$ with $P(1/p) \ne 0$.

For cohomology with coefficients, the formalism of pushforward and pullback maps works as follows: suppose that we have a closed immersion of PEL Shimura varieties $\iota: Y_H(U')\hookrightarrow Y_G(U)$ of codimension $d$, for some reductive group $H$ and  $U'=U\cap H(\Af)$. Assume that the closed immersion extends to the toroidal compactifications.  Let $W$ be a direct summand of $V|_H$. We then obtain
\begin{align}
(\iota_U^{\cW})_*:\, H^\star_{\NNfp}(Y_H(U'),\cW,n;P) & \rTo H^{\star+2d}_{\NNfp}(Y_G(U),\cV,d+n;P), \label{eq:Wpushforward}  \\
(\iota_U^{\cW})^*: \, H^{\star}_{\NNfp,c}(Y_G(U),\cV,n;P) & \rTo  H^\star_{\NNfp,c}(Y_H(U'),\cW,n;P) \label{eq:Wpullback}
\end{align}
for all $r\in\ZZ$ which are adjoint under the pairing \eqref{eq:NNfppairing}.


 \subsection{Comparison with log-rigid fp-cohomology}\label{sect:lrigfpcomp}

  Let $K$ be a finite extension of $\Qp$ with uniformizer $\pi$. Write $O_K^\pi$ for the scheme $\Spec O_K$ with the canonical log structure $1\mapsto \pi$. Let $(X,D)$ be a strictly semistable log scheme over $O_K^\pi$ with boundary, in the sense of \cite{ertlyamada18}, and write $U=X-D$. As explained in \cite[\S 9.]{LZ20}, one can define \emph{log-rigid fp-cohomology} (with or without compact support), denoted by $R\Gamma_{\lrigfp}(X\langle \pm D\rangle, n,\pi;P)$, for any $n\in\ZZ$. This cohomology theory is equipped with pullback and cup products, and it compares with NN-fp cohomology:

  \begin{proposition}
   If $X$ is proper, we have canonical isomorphisms
    	\begin{equation}\label{eq:NNlrigfp}
    	 R\Gamma_{\NNfp}(U_K,n;P)\cong R\Gamma_{\lrigfp}(X\langle  D\rangle, n,\pi;P)\quad \text{and}\quad R\Gamma_{\NNfp,c}(U_K,n;P)\cong R\Gamma_{\lrigfp}(X\langle - D\rangle, n,\pi;P).
  	    \end{equation}
  \end{proposition}

  We can also compare log-rigid fp cohomology to Besser's fp-cohomology of a smooth open subscheme $Z$ of $X$:

  \begin{proposition}\label{prop:extby0}
  	We have an extension-by-$0$ morphism
   \[ R\Gamma_{\rigfp,c}(Z\langle - D\rangle, n,\pi;P)\rTo R\Gamma_{\lrigfp}(X\langle - D\rangle, n,\pi;P)\]
   which is adjoint  to restriction to $Z$  under the cup product pairing.
  \end{proposition}

  \begin{remark}
   We define log-rigid fp-cohomology with coefficients in a filtered $F$-isocrystal $\sF$ using Liebermann's trick. This definition is compatible with NN fp-cohomology and Besser's log-rigid fp-cohomology  with coefficients under the maps above.
  \end{remark}

 We also need the notion of \emph{Gros rigid fp-cohomology} (with and without compact support), denoted  $\widetilde{R\Gamma}_{\rigfp}(Z\langle \pm D\rangle, n,\pi;P)$ and $\widetilde{R\Gamma}_{\rigfp,c}(Z\langle \mp D\rangle, n,\pi;P)$, respectively (see \cite[\S 9.2]{LZ20}). It is characterised by the fact that there are natural maps (specialisation, resp. cospecialisation)
 \begin{align*}
  \widetilde{R\Gamma}_{\rigfp}(Z\langle  \pm D\rangle, n,\pi;P) & \rTo R\Gamma_{\rigfp}(Z\langle \mp D\rangle, n,\pi;P)\\
  {R\Gamma}_{\rigfp,c}(Z\langle \mp D\rangle, n,\pi;P) & \rTo \widetilde{R\Gamma}_{\rigfp,c}(Z\langle \mp D\rangle, n,\pi;P)
 \end{align*}
  which are compatible with cup products.


\subsection{Reduction to log-rigid fp-cohomology of $Y_G$}\label{sect:redtoGSp4}

   We now apply the formalism described in Section \ref{sect:NNfp}  to the representations $V$, $\hat{V}$ and $\tilde{V}$. As in Proposition \ref{prop:branchingG}, let $r_1\geq r_2\geq 0$, $t_1,t_2\geq 0$ such that $r_1+r_2\equiv t_1+t_2\pmod 2$,
   \[ |t_1-t_2|\leq r_1-r_2\quad \text{and}\quad r_1-r_2\leq t_1-t_2\leq r_1+r_2.\]
   Let $q=\frac{1}{2}(r_1+r_2-t_1-t_2)$.

   \begin{proposition}\label{prop:pushpullmaps}
   	We then have pushforward maps
   	\begin{align*}
   		\iota^{[t_1,t_2]}_{1,*}: H^2_{\NNsyn}\left(Y_{H,\Delta}, \Sym^{t_1}\sH^\vee_{\QQ_p}\boxtimes\Sym^{t_2}\sH^\vee_{\QQ_p},2\right)&\rTo H^4_{\NNsyn}\left({Y}_{G,\Kl},\cV_{\QQ}^\vee,3-q\right),\\
 		\hat\iota^{[t_1]}_{1,*}:H^1_{\NNsyn}\left(Y_{H,\Delta}, \Sym^{t_1}\sH^\vee_{\QQ_p}\boxtimes \mathbf{1},1\right)&\rTo H^5_{\NNsyn}\left(\hat{Y}_{\Kl},\hat\cV_{\QQ_p}^\vee,3-q-t_2\right),\\
 		\tilde\iota_{1,*}:  H^0_{\NNsyn}\left(Y_H,\Qp\right)&\rTo H^6_{\NNsyn}\left(\tilde{Y}_{\Kl},\tilde\cV_{\QQ_p}^\vee,3-q-t_1-t_2\right)
   	\end{align*}
   	and pullback maps
   	\begin{align*}
 		(\iota^{[t_1,t_2]}_1)^*: \, H^{\star}_{\NNsyn,c}(Y_{G,\Kl},\cV_{\Qp},d+q)  &\rTo  H^\star_{\NNsyn,c}(Y_{H,\Delta},\Sym^{t_1}\sH_{\Qp}\boxtimes\Sym^{t_2}\sH_{\Qp},d),\\
     	(\hiota^{[t_1]}_1)^*: \, H^{\star}_{\NNsyn,c}(\hat{Y}_{G,\Kl},\hat\cV_{\Qp},d+q+t_2) & \rTo  H^\star_{\NNsyn,c}(Y_{H,\Delta},\Sym^{t_1}\sH_{\Qp}\boxtimes\mathbf{1},d),\\
 		\tilde\iota_1^*: H^{\star}_{\NNsyn,c}\left(\tilde{Y}_{\Kl},\tilde\cV_{\Qp},d+q+t_1+t_2\right) &\rTo H^\star_{\NNsyn,c}\left(Y_H,\Qp,d\right).
 		\end{align*}
  \end{proposition}

  \begin{lemma}\label{lem:relationsiota}
  	For $i=1,2$, let $x_i\in H^1_{\NNsyn}(Y_{\GL_2,\Gamma_0}, \Sym^{t_i}\sH^\vee,1)$, and let $y\in H^3_{\NNsyn,c}(Y_{G,\Kl},\cV,1+q)$. Then
  	\begin{align}
  	  \iota^{[t_1,t_2]}_{1,*}(x_1\sqcup x_2)\cup y & = \hat\iota^{[t_1]}_{\Delta,*}\left(x_1\sqcup \mathbf{1}\right) \cup \left(y\sqcup x_2\right)\\
  	  & = \tilde\iota_{1,*}(\mathbf{1}\sqcup \mathbf{1})\cup \left(y\sqcup x_1\sqcup x_2\right),
  	\end{align}
  	where we use the identification $\Sym^{t}\sH(-t)\cong \Sym^{t}\sH^\vee$.
  \end{lemma}
  \begin{proof}
  	Simple check.
  \end{proof}

  Write $\hat{D}$ for the boundary divisor $\hat{X}_{\Kl,\QQ}-\hat{Y}_{\Kl,\QQ}$.

 \begin{note}
 	We can regard $\eta_{\dR}$ as an element $\eta_{\dR,q}\in \Fil^{1+q}\DdR(V_p(\Pi))$.
 \end{note}

   \begin{lemma}\label{lem:etaNNfplift}
    Let $P_{q}(T)=P(p^{1+q}T)$. Then there exists a unique lift $\eta_{\NNfp,q,-D}$ of  $\eta_{\dR,q}$ to the group $H^3_{\NNfp}(X_{\Kl},\cV(-D),1+q;P_{q})[\Pi_f']$.
   \end{lemma}

  \begin{definition}\label{def:ginNNfp}
   For $i=1,2$ in case (C), and $i = 2$ in case (B), write $\mu_{i,\NNfp}$ for the unique $\Sif'$-equivariant lift of $\mu_{i,\dR}$ to \[H^1_{\NNfp}(Y_{\GL_2,\Gamma_0},\Sym^{t_i}\sH,1+t_i;Q_{i,t_i}),\] where $Q_i(T)=\left(1-\frac{T}{\fa_i}\right)\left(1-\frac{T}{\fb_i}\right)$.
  \end{definition}

  We now introduce syntomic Eisenstein classes for case (B). As in Section \ref{sect:periods}, let $(w,\Phi)$ be the product of  some arbitrary test data $(w^p, \Phi^p)$ away from $p$ and the Klingen test data at $p$. Shrinking $K^p$ if necessary, we may assume that $K^p$ fixes $w^p$, and $K_H^p$ fixes $\Phi^p$.
   \begin{notation}
    For $t\geq 0$, write $\Eis^{[t]}_{\syn,\Phi}$ for the image of $\Eis^{[t]}_{\mot,\Phi}$ under $r_{\NNsyn}$.
   \end{notation}

  \begin{proposition}\label{prop:likeGSp4}\
  	\begin{enumerate}
  		\item In case (B), we have
         \begin{align}
         & \left\langle \left(\log \circ \pr_{\hat\Pi} \circ \hiota^{[t_1]}_{i, \star}\right)(\Eis^{[t_1]}_{\et, \Phi^p \Phi_{\crit}}\sqcup \mathbf{1}),\eta_{\dR}\sqcup\mu_{2,\dR}\right\rangle_{\hY_{\Kl}}\notag\\
         &\qquad \qquad =\left\langle \Eis^{[t_1]}_{\syn, \Phi^p \Phi_{\crit}}\sqcup \mu_{2,\NNfp},\,  (\iota_1^{[t_1,t_2]})^*\left(\eta_{\NNfp,q, -D}\right)\right\rangle_{\NNfp,Y_{\Kl}},\label{eq:hatcase}
         \end{align}
        \item In case (C), we have
        \begin{align}
        & \left\langle \left(\log \circ \pr_{\tilde\Pi_f^{\prime \vee}} \circ \tilde\iota_{1, \star}\right)(\mathbf{1}\sqcup \mathbf{1}),\eta_{\dR}\sqcup\mu_{1,\dR}\sqcup\mu_{2,\dR}\right\rangle_{\hY_{\Kl}}\notag \\
        &\qquad \qquad =\left\langle\mu_{1,\NNfp}\sqcup \mu_{2,\NNfp},\,  (\iota_1^{[t_1,t_2]})^*\left(\eta_{\NNfp,q, -D}\right)\right\rangle_{\NNfp,Y_{\Kl}}.\label{eq:tildecase}
        \end{align}
   \end{enumerate}
 \end{proposition}
   \begin{proof}
  	 We prove (1). It is clear from the general formalism of Abel--Jacobi maps (c.f. \cite[\S 7.3]{LZ20}) that
  	 \begin{align*}
  	  &\left\langle \left(\log \circ \pr_{\hat\Pi_f^{\prime \vee}} \circ \hiota^{[t_1]}_{1, \star}\right)(\Eis^{[t_1]}_{\et, \Phi^p \Phi_{\crit}}\sqcup \mathbf{1}),\eta_{\dR}\sqcup\mu_{2,\dR}\right\rangle_{\hY_{\Kl}}\\
   	  & \qquad\qquad =  	 \left\langle (\hiota^{[t_1]}_1)_*( \Eis^{[t_1]}_{\syn, \Phi^p \Phi_{\crit}}\sqcup\mathbf{1}),\, \eta_{\NNfp,q, -D}\sqcup  \mu_{2,\NNfp}\right\rangle.
  	 \end{align*}
  	 The result now follows by applying Lemma \ref{lem:relationsiota} and the adjunction between pushforward and pullback.

  	 (2) can be proved analogously.
   \end{proof}

  Note that for $i = 1$ in case (B), the Eisenstein class lies in syntomic cohomology, which is fp-cohomology with polynomial $1 - T = 1 - p^{t_1 + 1} T / \fa_1$; we write $\Eis^{[t]}_{\fp,\Phi}$ for its image under the ``change of polynomial'' map in fp-cohomology with the quadratic polynomial $Q_{i, t_i} = (1 - p^{t_1 + 1} T / \fa_1)( 1 - p^{t_1 + 1} T / \fb_1)$ instead.
  \begin{notation}
  	In order to evaluate \eqref{eq:hatcase} and \eqref{eq:tildecase} simultaneously, we let
  	\[ x^{(1)}_{\NNfp}= \begin{cases} \Eis^{[t_1]}_{\fp, \Phi^p \Phi_{\crit}} &\quad \text{Case (B)}\\\mu^{(1)}_{\NNfp} & \quad \text{Case (C)}\end{cases}\]
  	 and  $x^{(2)}_{\NNfp}=\mu^{(2)}_{\NNfp}$. 
  \end{notation}

  Using the isomorphisms \eqref{eq:NNlrigfp} between NN-fp and lrig-fp cohomology, we obtain the following:

  \begin{lemma}\label{lem:redtolrigfp}
   We have
   \begin{align*}
    &\left\langle x^{(1)}_{\NNfp}\sqcup x^{(2)}_{\NNfp},\,  (\iota_{\Delta}^{[t_1,t_2]})^*(\eta_{\NNfp,q, -D})\right\rangle_{\NNfp,Y_{H,\Delta}}\\
    &\qquad =    \left\langle x^{(1)}_{\lrigfp}\sqcup x^{(2)}_{\lrigfp},\,  (\iota_{\Delta}^{[t_1,t_2]})^*(\eta_{\lrigfp,q, -D})\right\rangle_{\lrigfp,Y_{H,\Delta}}.
   \end{align*}
  \end{lemma}


\section{Comparison with the regulator evaluation for $\GSp(4)$}

 \subsection{Reduction to Gros fp-cohomology}

  The formula for the regulator in Lemma \ref{lem:redtolrigfp} resembles closely that of the 1st reduction step in the evaluation of the regulator for the $\GSp(4)$ Euler system: in \cite[Proposition]{LZ20}, the regulator is expressed as the pairing
  \begin{equation}\label{eq:redtoGrosGSp4}
      \left\langle \Eis^{[t_1]}_{\lrigsyn,\Phi_1}\sqcup \Eis^{[t_2]}_{\lrigsyn,\Phi_2}\,  (\iota_{\Delta}^{[t_1,t_2]})^*\left(\eta_{\lrigfp,q, -D}\right)\right\rangle_{\lrigfp,Y_{H,\Delta}},
  \end{equation}
  so all we have done is to replace the Eisenstein classes by $x^{(i)}_{\lrigfp}$. In \emph{op. cit.}, we reduce \eqref{eq:redtoGrosGSp4} to a pairing in Gros fp-cohomology over $Y^{\ord}_{H,\Delta}$. This reduction is independent of whether we pair against Eisenstein or cuspidal classes, so by the same argument, we obtain the analogous formula for our regulator. In order to state it, we need to recall some definitions.

  \begin{definition}
	Let $Y_{G,\Kl}$ be the canonical $\Zp$-model of the Siegel $3$-fold of level $K^{p} \times \Kl(p)$, for some (sufficiently small) tame level $K^p$. Let $X_{G,\Kl}$ be a toroidal compactification of $Y_{\Kl}$ (for some suitably chosen cone-decomposition $\Sigma$), and $X_{\Kl}^{\min}$ the minimal compactification. Write $D$ for the boundary divisor of the toroidal compactification. Write $Y_{G,\Kl,0}$ (resp. $X_{G,\Kl,0}$) for the special fibre of $Y_{G,\Kl}$ (resp. $X_{G,\Kl}$).

	We similarly write $Y_G$ for the canonical $\Zp$-model of the Siegel $3$-fold of level $K^{p} \times G(\Zp)$, and $X_G$ for its toroidal compactification.
 \end{definition}

 For $r\in\{0,1,2\}$, we have subschemes $X_{G,\Kl}^{\geq r}$ and $X_{G,\Kl}^{\leq r}$ of $X_{G,\Kl,0}$.

  \begin{notation}
	If $\cE$ is a coherent sheaf on $\cX^{\ge 1}_{\Kl}$,  write $R\Gamma_{c0}\left(\cX_{G,\Kl}^{\ord}, \cE\right)$ for the cohomology of $\cX_{G,\Kl}^{\ord}$ with coefficients in $\cE$, compactly supported away from $\cX_{G,\Kl}^{=1}$.
  \end{notation}

  \begin{definition}
	Let $X_{H,\Delta}$ be the covering of $X_H$ parametrising choices of finite flat $\Zp$-subgroup-scheme $C \subset (E_1 \oplus E_2)[p]$ of order $p$ which project nontrivially into both $E_i$.
\end{definition}

\begin{lemma}
	There is a finite map $\iota_1:X_{H, \Delta} \to X_{ \Kl}$ lying over the map $X_H \to X_G$, and the preimage of $X_{G,\Kl}^{\ge 1}$ is the open subscheme $X_{H, \Delta}^{\ord}$ where $C$ is multiplicative.
\end{lemma}
\begin{proof}
	See  \cite[\S 4.1]{LPSZ1}.
\end{proof}

We therefore obtain a finite map of dagger spaces $\iota_1:\cX_{H, \Delta}^{\ord} \to \cX_{G, \Kl}^{\ge 1}$ (whose image is contained in $\cX_{G, \Kl}^{\ord}$).

\begin{proposition}\label{prop:pullbackfactor}
	For any locally free coherent sheaf $\cE$ on $\cX_{\Kl}^{\ge 1}$, the natural pullback map
	\[ \iota_1^*: R\Gamma_c(\cX_{G, \Kl}^{\ge 1}, \cE) \to R\Gamma_c(\cX_{H, \Delta}^{\ord}, \iota^*\cE)\]
	factors through $R\Gamma_{c0}(\cX_{\Kl}^{\ord}, \cE)$.
\end{proposition}
\begin{proof}
	This is \cite[Prop. 10.3.4]{LZ20}.
\end{proof}

\begin{proposition}
	Assuming Conjecture 10.2.3 in \cite{LZ20}, the extension-by-0 map
\[  H^3_{\rigfp, c}(X_{G,\Kl}^{\ge 1}\langle -D\rangle, \cV, n, P)\to  H^3_{\lrigfp}(X_{G,\Kl}\langle -D\rangle, \cV, n, P). \]
	induces an isomorphism on the $\Pif'$-eigenspace.
\end{proposition}
\begin{proof}
	See \cite[Prop. 10.4.3]{LZ20}.
\end{proof}

   \begin{corollary}
	There is a unique class
	\[ \eta_{\rigfp,q, -D}^{\ge 1} \in H^3_{\rigfp, c}(X_{G,\Kl}^{\ge 1}\langle -D\rangle, \cV, 1+q, P)[\Pif'] \]
	which is in the $\Pif$-eigenspace for the prime-to-$p$ Hecke operators and maps to $\eta_{\dR, q,-D}$ under extension-by-0.
   \end{corollary}

   \begin{note}
   	As shown in \cite[Prop. 12.3.1]{LZ20}, $\eta_{\rigfp,q, -D}^{\ge 1}$ lifts uniquely to an element in \emph{Gros fp-cohomology} (c.f. \cite[Prop. 9.2.20]{LZ20})
   	\[ \tilde{\eta}_{\rigfp,q, -D}^{\ge 1} \in \wH^3_{\rigfp, c}(\cX_{G,\Kl}^{\ge 1}\langle -\cD\rangle, \cV, 1+q; P_{q}), \]
   \end{note}

   \begin{definition}
   	\begin{enumerate}
   	 \item Write  $\tilde{\eta}_{\rigfp,q, -D}^{\ord}$ for the image of $ \tilde{\eta}_{\rigfp,q, -D}^{\ge 1} $ in
   	 $\wH^3_{\rigfp, c0}(\cX_{G,\Kl}^{\ord}\langle -\cD\rangle, \cV, 1+q; P_{q})$ under restriction to the ordinary locus. Here, $\wH^\bullet_{\rigfp, c0}(\cX_{G,\Kl}^{\ord},\,\sim\,)$ denotes Gros fp-cohomology of $\cX_{G,\Kl}^{\ord}$ with compact support towards $\cX_{G,\Kl}^{=0}$, as defined in \S 9.2.4 and Def. 12.2.1 in \emph{op.cit.}
	 \item  Write $ \tilde{\eta}_{\rigfp,q}^{\ord}$ for the image of  $\tilde{\eta}_{\rigfp,q, -D}^{\ord}$  in $\wH^3_{\rigfp, c0}(\cX_{G,\Kl}^{\ord}, \cV, 1+q; P_{q})$ under the `forget $-\cD$' map.
	\end{enumerate}
   \end{definition}

\begin{notation}
	For $i=1,2$, write $x^{(i)}_{\rigfp,\ord}$ for the image of $x^{(i)}_{\rigfp}$ in $H^1_{\rigfp}(Y_{\GL_2,\Gamma_0(p)}^{\ord},\Sym^{t_i}\sH_{\Qp},1+t_i;R_{i,t_i})$ under the restriction map $\res_{Y_{\GL_2,\Gamma_0(p)}^{\ord}}$. Denote by
	\[   \tilde{x}^{(i)}_{\rigfp,\ord}\in \wH^1_{\rigfp}(\cY_{\GL_2,\Gamma_0}^{\ord},\Sym^{t_i}\sH,1+t_i;S_{i,t_i})\]
	the image\footnote{More precisely, the image lies in fp-cohomology with $Q_{i, t_i}$, but this image lifts to a class with this simpler polynomial.} of $x^{(i)}_{\rigfp,\ord}$ in the Gros-fp cohomology of $\cY_{\GL_2,\Gamma_0}^{\ord}$. Here, we let 
	\[ S_i(T)=1-\frac{T}{\mathfrak{b}_i}.\]
\end{notation}

   Arguing as in \cite[\S 12.3]{LZ20}, we obtain the following result:

   \begin{proposition}\label{prop:restoGros}(c.f. Cor. 12.3.6 in \emph{op.cit.})
 	We have
 	\begin{align}
		 & \left\langle x^{(1)}_{\lrigfp}\sqcup x^{(2)}_{\lrigfp},\,  (\iota_{\Delta}^{[t_1,t_2]})^*(\eta_{\lrigfp,q, -D})\right\rangle_{\lrigfp,Y_{H,\Delta}}\\
 	  & \qquad \qquad = \left\langle \tilde{x}^{(1)}_{\rigfp,\ord}\sqcup  \tilde{x}^{(2)}_{\rigfp,\ord},\, (\iota_\Delta^{[t_1,t_2]})^*( \tilde{\eta}_{\rigfp, q,-D}^{\ord}|_{Y^{\ord}_{\Kl}})\right\rangle_{\widetilde{\rigfp},Y^{\ord}_{H,\Delta}}.\label{eq2}
 	\end{align}
   \end{proposition}


 \subsection{Representation of the classes as coherent $\fp$-pairs}

  \subsubsection{Classes from $\Pi$}

    As in \cite[\S 5.2]{LPSZ1}, the pair $(r_1, r_2)$ determines algebraic representations $L_i$ of the Siegel Levi $M_{\Sieg}$, for $0 \le i \le 3$, all with central character $\diag(x, \dots, x) \mapsto x^{r_1 + r_2}$; and hence vector bundles $\cL_i = [L_i]_{\can}$ on $X_{K, \QQ}$ for any sufficiently small level $K$ (the canonical extensions of the corresponding vector bundles over $Y_{K, \QQ}$). Let $N^i = L_{3-i}$, and $\cN^i = \cL_{3-i}$ the corresponding vector bundles.

   \begin{proposition}
	\label{prop:etaproperties}
	There exists a unique class $\eta^{\ge 1}_{\coh,-D} \in H^2_c\left(\cX_{G,\Kl}^{\ge 1}, \cN^1(-D)\right)$ with the following two properties:
	\begin{enumerate}
		\item $U'_{\Kl, 2}$ acts on $\eta^{\ge 1}_{\coh,-D}$ as multiplication by  $\frac{\alpha\beta}{p^{r_2 + 1}}$.
		\item The image of $\eta^{\ge 1}_{\coh,-D}$ under the extension-by-zero map is $\eta^{\mathrm{alg}}_{-D}$.
	\end{enumerate}
	This class enjoys the following additional properties:
	\begin{enumerate}
		\item[(3)] The operator $U'_{\Kl, 1}$ acts on $\eta^{\ge 1}_{\coh,-D}$ as multiplication by $\alpha + \beta$.
		\item[(4)] The spherical Hecke algebra acts via the system of eigenvalues associated to $\Pi'$.
	\end{enumerate}
   \end{proposition}
   \begin{proof}
   	See \cite[Prop. 11.6.3]{LZ20}.
   \end{proof}

   \begin{definition}\
	 Define $\breve{\eta}^{\geq 1}_{\coh,q,-D}$ to be the image of $\eta^{\geq 1}_{\coh,-D}$ under the composition of maps
	\[  H^2_{c}(\cX_{G,\Kl}^{\geq 1}, \cN^1\langle -\cD\rangle)\rTo H^2_{c}(\cX_{G,\Kl}^{\geq 1}, \sFil^{r_2}\cV\otimes\Omega_G^1\langle -D \rangle)\rTo H^2_{c}(\cX_{G,\Kl}^{\geq 1}, \sFil^{q}\cV\otimes\Omega_G^1\langle -D \rangle),\]
	where the first map is given by the inclusion of complexes, and the second map is induced from the natural inclusion of sheaves.
   \end{definition}

  As shown in Prop. 12.3.7 in \emph{op.cit.}, every element in $\wH^3_{\rigfp, c0}(\cX_{G,\Kl}^{\ord}, \cV, 1+q; P_{q})$ can be uniquely represented by a pair of classes
  \[ (x,y)\in H^2_{c0}(\cX_{G,\Kl}^{\ord},\cN^0)\oplus H^2_{c0}(\cX_{G,\Kl}^{\ord},\cN^1)\]
  which satisfies $\nabla(x)=P_q(\varphi)\cdot y$ and $\nabla(y)=0$.

 \begin{lemma}(c.f. Proposition 12.4.5 in \emph{op.cit})
  The class $ \tilde{\eta}_{\rigfp,q}^{\ord}$ is represented by the pair of classes $\left( \breve{\eta}^{\ord}_{\coh,q},\breve\zeta\right)$, with $\breve\zeta\in  H^2_{c0}(\cX^{\ord}_{G,\Kl},\cV\otimes \Omega^0_G\langle \cD\rangle)$, which satisfy
   \[ \nabla \breve{\eta}^{\ord}_{\coh,q}=0\qquad\text{and}\qquad P_{q}(\Phi_{1+q})\,\breve{\eta}^{\ord}_{\coh,q}=\nabla\breve{\zeta}.\]
 \end{lemma}


\subsubsection{Classes from $\Sigma_i$}

 We have a similar (but much simpler) result for $ \tilde{x}^{(i),\ord}_{\rigfp}$.

 \begin{lemma}
 	The class $\tilde{x}^{(i),\ord}_{\rigfp}$ lifts to an element
    \[ \tilde{x}^{(i),\ord}_{\rigfp,-D_{\GL_2}} \in \wH^1_{\rigfp}(\cX_{\GL_2,\Gamma_0}^{\ord}\langle -\cD_{\GL_2}\rangle,\Sym^{t_i}\sH_{\Qp},1+t_i;S_{i,t_i}).\]
    Here,  $S_i(T)=1-\mathfrak{b}_i^{-1}T$, where $\mathfrak{b}_1$ is as defined in Definition \ref{def:aibiEis} in case (B).
 \end{lemma}
 \begin{proof}
	 There is nothing to prove here except in case (B) for $i=1$ where the result follows from the fact that $S_{i, t_i}(\varphi)$ annihilates the constant term of the Eisenstein series at ordinary cusps. This is the same argument as in \cite[\S 15.4]{LZ20} (the ``herb-chopper'' diagram).
 \end{proof}

 Let $u,\, v$ denote the basis of sections $\tilde{\omega}$ and $\tilde{u}$ over the Igusa tower, as constructed in \cite[\S 4.5]{KLZ20} (c.f. also \cite[\S 15.5]{LZ20}).

 \begin{definition}\
 	\begin{itemize}
 		\item  Let $g_i$ be the Eisenstein series $F^{t_1+2}_{\Phi^{(p)} \Phi_{\crit}}$ in case (B) for $i=1$, and let it be the normalized cuspidal eigenform in $\Sigma_i$ of $p$-level $\Gamma_0(p)$  with $U_p$-eigenvalue $\mathfrak{a}_i$ otherwise.
 		\item Let $G_i$ be the $p$-adic Eisenstein series $E^{-t_1}_{\Phi^{(p)} \Phi_{\dep}}$ in case (B) for $i=1$, and let it be the unique $p$-adic modular form with trivial constant term of weight $-t_i$ such that $\theta^{t_i+1}G_i=g_i^{[p]}$ otherwise.
 	\end{itemize}
 \end{definition}

 \begin{note}
	These definitions are consistent with each other, since in case (B) the Eisenstein series $F^{t_1+2}_{\Phi^{(p)} \Phi_{\crit}}$ is in the $U_p=p^{t_1+1}$ eigenspace and we have defined $\fa_1 = p^{t_1+1}$, so in both cases $g_i$ is a $U_p = \fa_i$ eigenvector.
 \end{note}

 \begin{lemma}
 	The class $ \tilde{x}^{(i),\ord}_{\rigfp,-D_{\GL_2}} $ is represented uniquely by the pair $(\epsilon_0^{(i)},\epsilon_1^{(i)})$ with
 	\[ \epsilon_0^{(i)}=\sum_{j=0}^{t_i}\frac{(-1)^jt_i!}{(t_i-j)!}\, \theta^jG_i\cdot v^{t_i-j}w^j \quad \text{and} \quad \epsilon_1^{(i)}=g_i\cdot v^{t_i}\otimes \xi,\]
 	where $\xi$ is as defined in \cite[\S 4.5]{KLZ20}.
 \end{lemma}

 \begin{note}
  We have $U_p(\epsilon_0^{(i)})=0$.
 \end{note}

  \begin{lemma}
	The class $ \tilde{x}^{(1)}_{\rigfp,\ord,-D_{\GL_2}}\sqcup  \tilde{x}^{(2)}_{\rigfp,\ord,-D_{\GL_2}}$ is represented by the coherent fp-pair
	\[ \left( \alpha_1,\alpha_2\right)\in H^{2,1}\left(\cX^{\ord}_{H,\Delta},\Sym^{t_1}\sH_{\Qp}\boxtimes \Sym^{t_2}\sH_{\Qp},2+t_1+t_2;S\right),\]
	where $S(y)=1-p^{2+t_1+t_2}\mathfrak{b}_1^{-1}\mathfrak{b}_2^{-1}\,y$ and
	\begin{align}
	\alpha_1 &= \epsilon^{(1)}_0\sqcup\epsilon^{(2)}_1 + \mathfrak{b}_1^{-1}p^{1+t_1}(  \varphi_{\GL_2}^*\boxtimes 1)\, \left(\epsilon^{(1)}_1\sqcup \epsilon^{(2)}_0\right),\\
	\alpha_2 & =  \epsilon^{(1)}_1\sqcup \epsilon^{(2)}_1.
	\end{align}
\end{lemma}


 \subsection{Evaluation of the pairing}

  We can now evaluate the pairing \eqref{eq2}. As a first step, we reduce it to a pairing in coherent cohomology, analogous to \cite[Lemma 6.1.1]{LZ20}:

  \begin{lemma}
  	We have
  	\begin{align}
  	  & \left\langle \tilde{x}^{(1)}_{\rigfp,\ord}\sqcup  \tilde{x}^{(2)}_{\rigfp,\ord},\, (\iota_1^{[t_1,t_2]})^*( \tilde{\eta}_{\rigfp, q,-D}^{\ord}|_{Y^{\ord}_{\Kl}})\right\rangle_{\widetilde{\rigfp},Y^{\ord}_{H,\Delta}}\notag\\
  	  & \qquad   =  -\left\langle  \left(\iota^{[t_1,t_2]}_1\right)^*(\breve\zeta,\breve\eta^{\ord}_{\coh,q}),\,  \left( \alpha_1, \alpha_2\right)\right\rangle_{\coh-\mathrm{fp},\cX^{\ord}_{H,\Delta}}.\label{cohpairing}
  	\end{align}
  \end{lemma}

  We evalute this expression using Besser's formalism for cup products in fp-cohomology:  let
  \[ a(x,y)=p^{t_1+t_2+2}(\mathfrak{b}_1\mathfrak{b}_2)^{-1}\, y\]
  and
  \[ b(x,y)=\frac{P_q\left(p^{t_1+t_2+2}(\mathfrak{b}_1\mathfrak{b}_2)^{-1}\, xy\right)-p^{t_1+t_2+2}(\mathfrak{b}_1\mathfrak{b}_2)^{-1}\,yP_q(x)}{1-p^{t_1+t_2+2}(\mathfrak{b}_1\mathfrak{b}_2)^{-1}\,y},\]
  so we have
  \[ P_q\star R(xy)=a(x,y)P_q(x)+b(x,y)\left(1-p^{t_1+t_2+2}(\mathfrak{b}_1\mathfrak{b}_2)^{-1}\,y\right).\]
  Let $\Upsilon=\left(1-\frac{p^h}{\alpha\mathfrak{b}_1\mathfrak{b}_2}\right)\left(1-\frac{p^h}{\beta\mathfrak{b}_1\mathfrak{b}_2}\right)$.

  Then
  \begin{align*}
  & \Upsilon\times \left\langle  \left(\iota^{[t_1,t_2]}_1\right)^*(\breve\zeta,\breve\eta^{\ord}_{\coh,q}),\,  \left( \alpha_1, \alpha_2\right)\right\rangle_{\coh-\mathrm{fp},\cX^{\ord}_{H,\Delta}}\\
    &\qquad  =  \ a(\varphi_{H,1}^*\otimes 1,1\otimes \varphi_{H}^*)\, \left[(\iota^{[t_1,t_2]}_1)^*(\zeta)\cup \alpha_2\right]  + b(\varphi_{H,1}^*\otimes 1,1\otimes \varphi_{H}^*)\, \left[(\iota^{[t_1,t_2]}_1)^*(\eta^{\ord}_{\coh})\cup \alpha_1\right]\notag \\
    & \qquad = b(\varphi_{H,1}^*\otimes 1,1\otimes \varphi_H^*)\, \left[(\iota^{[t_1,t_2]}_1)^*(\breve\eta^{\ord}_{\coh,q})\cup \alpha_1\right],
  \end{align*}
  where the last equality follows by the same argument as in the proof of \cite[Prop. 16.1.2]{LZ20}.

  \begin{remark}
  	The factor $\Upsilon$ arises from the  normalisation of the trace map on finite-polynomial cohomology.
  \end{remark}

    Write $P(x)=1+c_1x+c_2x^2$; by definition, we have $c_2=(\alpha\beta)^{-1}$ and $c_1=-\frac{\alpha+\beta}{\alpha\beta}$. Then
    \[ b(x,y)=1-c_2\,p^{t_1+t_2+2}(\mathfrak{b}_1\mathfrak{b}_2)^{-1}\, x^2y.\]
   We  now identify $\varphi_{H,1}^*$ with $p^{-1}\varphi_H^*$. Expanding the terms, we obtain the following expression:

    \begin{corollary}
	We have
	\begin{align}
	\text{\eqref{eq2}} =\quad  & (\iota^{[t_1,t_2]}_1)^*(\breve\eta^{\ord}_{\coh,q})\cup \alpha_1-c_2\,p^{t_1+t_2}(\mathfrak{b}_1\mathfrak{b}_2)^{-1}\cdot \varphi_H^*\left[(\iota^{[t_1,t_2]}_1)^*\varphi_H^*(\breve\eta^{\ord}_{\coh,q})\cup \alpha_1\right]\notag\\
	=\quad&   (\iota^{[t_1,t_2]}_\Delta)^*(\breve\eta^{\ord}_{\coh,q})\cup \left(\epsilon^{(1)}_0\sqcup\epsilon_1^{(2)} \right) \label{1stterm}\\
	&  -c_2\,p^{t_1+t_2}(\mathfrak{b}_1\mathfrak{b}_2)^{-1}\, \varphi_H^*\left[(\iota^{(t_1,t_2)}_1)^*\varphi_H^*(\breve\eta^{\ord}_{\coh,q})\cup \left(\epsilon^{(1)}_0\sqcup \epsilon^{(2)}_1 \right)\right]\label{2ndterm}\\
	& +p^{t_1}\mathfrak{b}_1^{-1}\, (\iota^{[t_1,t_2]}_1)^*(\breve\eta^{\ord}_{\coh,q})\cup  \left(\varphi_{\GL_2}^*\epsilon^{(1)}_1\sqcup\epsilon_0^{(2)}\right)\label{3rdterm}\\
	& -c_2\,\mathfrak{b}_1^{-2}\mathfrak{b}_2^{-1}p^{2t_1+t_2+1}\left(\langle p\rangle^{-1}\boxtimes\mathbf{1}\right)\,\varphi_H^*\left[ (\iota^{[t_1,t_2]}_1)^*\varphi_H^*(\breve\eta^{\ord}_{\coh,q})\cup  \left(\varphi_{\GL_2}^*\epsilon^{(1)}_0\sqcup\epsilon_0^{(2)}\right)\right].\label{4thterm}
	\end{align}
   \end{corollary}

  \begin{lemma}
     We have
     \begin{align*}
       \varphi_H^*\left[(\iota^{(t_1,t_2)}_1)^*\varphi_H^*(\breve\eta^{\ord}_{\coh,q})\cup \left(\epsilon^{(1)}_0\sqcup \mu_{\coh,-D_{\GL_2}}^{\ord} \right)\right]&=0\\
      \varphi_H^*\left[ (\iota^{[t_1,t_2]}_1)^*\varphi_H^*(\breve\eta^{\ord}_{\coh,q})\cup  \left(\varphi_{\GL_2}^*\epsilon^{(1)}_0\sqcup\epsilon_0^{(2)}\right)\right]&=0
    \end{align*}
  \end{lemma}
  \begin{proof}
  	Analogous to the proof of \cite[Lemma 16.1.4]{LZ20}.
  \end{proof}

  We hence deduce the following formula for the pairing:

   \begin{proposition}\label{prop:red2}
	We have
	\[
	\text{\eqref{eq2}}=  (\iota^{[t_1,t_2]}_\Delta)^*(\breve\eta^{\ord}_{\coh,q})\cup \left(\epsilon^{(1)}_0\sqcup\epsilon_1^{(2)} \right)
	 +p^{t_1}\, \mathfrak{b}_1^{-1}\, (\iota^{[t_1,t_2]}_1)^*(\breve\eta^{\ord}_{\coh,q})\cup  \left(\varphi_{\GL_2}^*\epsilon^{(1)}_1\sqcup\epsilon_0^{(2)}\right).
	\]
   \end{proposition}

    We now apply \cite[Cor. 14.2.4]{LZ20}.

\begin{note}\label{lem:Grbasis}
	For $0\leq \ell\leq t_1+t_2$, a basis of $\Gr^\ell V_H$ is given by
	\[\{ v^{t_1-i}w^{i}\boxtimes v^{t_2-j}w^{j}: 0\leq i\leq t_1,\, 0\leq j\leq t_2,\, i+j=t_1+t_2-\ell\}.\]
\end{note}

\begin{lemma}\label{lem:EsynimageinGr1}
	The image of $\epsilon^{(1)}_0\sqcup\epsilon_1^{(2)}$ under projection to $\Gr^{r_1-q}\left(\Sym^{t_1}\sH\boxtimes \Sym^{t_2}\sH\right)$ is given by
	\[ (-1)^{r_2-q}\frac{t!}{(q-r_2+t_1)!}\times \theta^{(q-r_2+t_1)}G_1\cdot  \,v^{r_1-r_2-r}w^{r_2-q}\boxtimes g_2\cdot (v^{t_2}\otimes \xi\otimes e_1).\]
\end{lemma}

\begin{lemma}
	\label{lem:EsynimageinGr2}
	The image of $\varphi_{\GL_2}^*\epsilon^{(1)}_1\sqcup\epsilon_0^{(2)}$ in $\Gr^{r_1-q}\left(\Sym^{t_1}\sH\boxtimes \Sym^{t_2}\sH\right)$ is given by
	\[
	(-1)^{r_2-q}\frac{t_2!}{(q-r_2+t_2)!}\times \left( \varphi_{\GL_2}^*g_1\cdot\,(v^{t}\otimes \xi\otimes e_1)\right)\boxtimes \left(\theta^{q-r_2+t_2}G_2\cdot v^rw^{r_2-q}\right).
	\]
\end{lemma}

   \begin{proposition}
	\label{prop:linktoLPSZ}
	We have
	\begin{align*}
	(\iota^{[t_1,t_2]}_\Delta)^*&(\breve\eta^{\ord}_{\coh,q})\cup \left(\epsilon^{(1)}_0\sqcup\epsilon_1^{(2)}\right)\\
	& =\frac{(-1)^{r_2-q}\,t_1!}{(q-r_2+t_1)!}\times \frac{(-2)^q}{\tbinom{t_1}{r_2-q}}\times \left\langle \eta^{\ord}_{\coh,q},\,\iota^{p-adic}_\star\left( \theta^{q-r_2-1}g^{[p]}_1\sqcup g_2\right)\right\rangle\\
	& = (-1)^{r_2-q}(-2)^q\, (r_2-q)!\times \left\langle \eta^{\ord}_{\coh,q},\,\iota^{p-adic}_\star\left( \theta^{q-r_2-1}g^{[p]}_1\sqcup g_2\right)\right\rangle,
	\end{align*}
	and
	\begin{align*}
	&(\iota^{[t_1,t_2]}_\Delta)^*(\breve\eta^{\ord}_{\coh,q})\cup \left(\varphi_{\GL_2}^*\epsilon^{(1)}_1\sqcup\epsilon_0^{(2)}\right)\\
	&\quad =\frac{(-1)^{r_2-q}\, t_2!}{(q-r_2+t_2)!}\times \frac{(-2)^q}{\tbinom{t_2}{r_2-q}}\times \left\langle \eta^{\ord}_{\coh,q},\,\iota^{p-adic}_\star\left(\varphi_{\GL_2}^*g_1\boxtimes \theta^{q-r_2-1}g^{[p]}_2\right)\right\rangle\\
	& \quad = (-1)^{r_2-q}(-2)^q(r_2-q)!\times \left\langle \eta^{\ord}_{\coh,q},\,\iota^{p-adic}_\star\left(\varphi_{\GL_2}^* g_1\boxtimes \theta^{q-r_2-1}g^{[2]}_2\right)
	\right\rangle.
	\end{align*}
  \end{proposition}
  \begin{proof}
  	Analogous to the proof of \cite[Prop. 16.9.1]{LZ20}
  \end{proof}

  To relate the formula of \cref{prop:linktoLPSZ} to the $p$-adic coherent periods of Proposition \ref{prop:weprove}, we use the following useful lemma.

  \begin{proposition}\label{prop:weirdid}
   If $k_1, k_2$ are any integers (not necessarily positive) with $k_1 + k_2 = r_1 - r_2 + 2$, and $\cG_1, \cG_2$ are $p$-adic modular forms of weights $k_i$ such that $U_p(\cG_1) = 0$, then
   \[ \Big\langle \eta^{\ord}_{\coh}, \iota_{\star}\left( \cG_1 \boxtimes \varphi(\cG_2) \right)\Big\rangle = 0.\]
  \end{proposition}

  \begin{proof}
   We shall prove this via a modification of the argument used in Corollary 12.5.3 of \cite{LZ20} (which was a related vanishing statement involving a ``$p$-depleted'' class on $G$, rather than on $H$ as here). We use the following identity of correspondences $X_{H, \Delta}^{\ord}(p) \rightrightarrows X_{G, \Kl}^{\ord}(p)$:
   \[ Z' \circ \iota_1 \circ (\varphi, R_p) = U_2' \circ \iota_1 \circ (1, U_p).\]
   Note that composing this with $(U_p, 1)$ gives the identity used \emph{loc.cit.}, and the new, slightly stronger identity can be proved by exactly the same argument.

    Since $\eta^{\ord}_{\coh}$ is an eigenvector for $U_2'$, with non-zero eigenvalue, we can argue that
   \begin{gather*}
    \Big\langle \eta, \iota_{1, \star}\left( \cG_1 \boxtimes \varphi(\cG_2) \right)\Big\rangle =
    \tfrac{p^{r_2 + 1}}{\alpha\beta} \Big\langle U_2' \cdot \eta, \iota_{1, \star}\left( \cG_1 \boxtimes \varphi(\cG_2) \right)\Big\rangle_{X_{G, \Kl}^{\ge 1}} \\
    = \tfrac{p^{r_2 + 1}}{\alpha\beta} \Big\langle ((1, U_p) \circ \iota_1^* \circ U_2') \, \eta, \cG_1 \boxtimes \cG_2 \Big\rangle_{X_{H, \Delta}^{\ord}} \\
    = \tfrac{p^{r_2 + 1}}{\alpha\beta} \Big\langle ((\varphi, p^{k_2 - 2} \langle p \rangle) \circ \iota_1^* \circ Z') \, \eta, \cG_1 \boxtimes \cG_2\Big\rangle_{X_{H, \Delta}^{\ord}}\\
    = \tfrac{p^{r_2 + 1}}{\alpha\beta} \Big\langle (\iota_1^* \circ Z')\, \eta, \left(U_p(\cG_1) \boxtimes p^{k_2-2} \langle p \rangle\cG_2\right)\Big\rangle_{X_{H, \Delta}^{\ord}} = 0.\qedhere
   \end{gather*}
  \end{proof}

  \begin{remark}
  	There is a similar statement with the roles of $\mathcal{G}_1$ and $\cG_2$ interchanged.
  \end{remark}

  \begin{corollary}
  	We have
  	\[ \text{\eqref{eq2}}=(-1)^{r_2-q}(-2)^q(r_2-q)!\times\left\langle \eta^{\ord}_{\coh,q},\,\iota^{p-adic}_\star\left( \theta^{q-r_2-1}g^{[p]}_1\sqcup g_2^{[p]}\right)\right\rangle\]
  	and hence
  	\begin{align*}
  	 &\left\langle \tilde{x}^{(1)}_{\rigfp,\ord}\sqcup  \tilde{x}^{(2)}_{\rigfp,\ord},\, (\iota_1^{[t_1,t_2]})^*( \tilde{\eta}_{\rigfp, q,-D}^{\ord}|_{Y^{\ord}_{\Kl}})\right\rangle_{\widetilde{\rigfp},Y^{\ord}_{H,\Delta}}\\
  	 &\qquad =\frac{(-1)^{r_2-q+1}(-2)^q(r_2-q)!}{\left(1-\frac{p^h}{\alpha\mathfrak{b}_1\mathfrak{b}_2}\right)\left(1-\frac{p^h}{\beta\mathfrak{b}_1\mathfrak{b}_2}\right)}\times\left\langle \eta^{\ord}_{\coh,q},\,\iota^{p-adic}_\star\left( \theta^{q-r_2-1}g^{[p]}_1\sqcup g_2^{[p]}\right)\right\rangle.
  	\end{align*}
  \end{corollary}
  \begin{proof}
  	Applying Proposition \ref{prop:weirdid} allows us to simplify both of the terms in Proposition \ref{prop:linktoLPSZ}: firstly, we have
  	\[ \left\langle \eta^{\ord}_{\coh,q},\,\iota^{p-adic}_\star\left(\varphi_{\GL_2}^* g_1\boxtimes \theta^{q-r_2-1}g^{[2]}_2\right)\right\rangle = 0;\]
  	secondly, we have
  	\[ \left\langle \eta^{\ord}_{\coh,q},\,\iota^{p-adic}_\star\left( \theta^{q-r_2-1}g^{[p]}_1\sqcup g_2\right)\right\rangle = \left\langle \eta^{\ord}_{\coh,q},\,\iota^{p-adic}_\star\left( \theta^{q-r_2-1}g^{[p]}_1\sqcup g_2^{[p]}\right)\right\rangle\]
  	since $g_2 - g_2^{[p]} = \tfrac{1}{\fb_2} \varphi^*(g_2)$.
  \end{proof}

  Arguing as in \cite[\S 16.5]{LZ20} completes the proof of Proposition \ref{prop:weprove}.


%

\newlength{\bibitemsep}
\setlength{\bibitemsep}{0.75ex plus 0.05ex minus 0.05ex}
\newlength{\bibparskip}
\setlength{\bibparskip}{0pt}
\let\oldthebibliography\thebibliography
\renewcommand\thebibliography[1]{%
 \oldthebibliography{#1}%
 \setlength{\parskip}{\bibparskip}%
 \setlength{\itemsep}{\bibitemsep}%
}

\newcommand{\noopsort}[1]{\relax}

\providecommand{\bysame}{\leavevmode\hbox to3em{\hrulefill}\thinspace}
\providecommand{\MR}[1]{}
\renewcommand{\MR}[1]{%
 MR \href{http://www.ams.org/mathscinet-getitem?mr=#1}{#1}.
}
\providecommand{\href}[2]{#2}
\newcommand{\articlehref}[2]{\href{#1}{#2}}

\end{document}